\documentclass[]{amsart}
\usepackage{amssymb,
            amsthm,
            amsmath,
            mathtools,
            wasysym,
            bbm,
            enumitem,
            lmodern,
            }
\usepackage[bookmarks = false]{hyperref}            
            
\usepackage[usenames,dvipsnames]{color}
\usepackage{MnSymbol}
\hypersetup{
 colorlinks,
 linkcolor={red!50!black},
 citecolor={blue!50!black},
 urlcolor={red!50!black}
}

\usepackage[%
    backend = biber,
    style = alphabetic,
    citestyle = alphabetic, isbn = false,
    doi = false,
    url = false,
    ]{biblatex} 
\usepackage[a4paper, inner=4.3cm,outer=4.3cm,top=3.3cm,bottom=3cm]{geometry}
\usepackage{comment}
\usepackage{subcaption}
\usepackage{graphicx}
\usepackage{pgfplots}
\usepackage{tikz,tikz-3dplot,tikz-cd}
\usetikzlibrary{cd}
\usetikzlibrary{shapes.misc}
\usetikzlibrary{decorations.markings}
\usepackage{xifthen}
\usepackage{array}
\usepackage[new]{old-arrows}

\AtBeginDocument{%
\def\MR#1{}
}

\makeatletter
\def\@tocline#1#2#3#4#5#6#7{\relax
  \ifnum #1>\c@tocdepth 
  \else
    \par \addpenalty\@secpenalty\addvspace{#2}%
    \begingroup \hyphenpenalty\@M
    \@ifempty{#4}{%
      \@tempdima\csname r@tocindent\number#1\endcsname\relax
    }{%
      \@tempdima#4\relax
    }%
    \parindent\z@ \leftskip#3\relax \advance\leftskip\@tempdima\relax
    \rightskip\@pnumwidth plus4em \parfillskip-\@pnumwidth
    #5\leavevmode\hskip-\@tempdima
      \ifcase #1
       \or\or \hskip 1em \or \hskip 2em \else \hskip 3em \fi%
      #6\nobreak\relax
    \hfill\hbox to\@pnumwidth{\@tocpagenum{#7}}\par
    \nobreak
    \endgroup
  \fi}
\makeatother

\theoremstyle{definition}
\newtheorem{thm}{Theorem}[section]
\newtheorem{prop}[thm]{Proposition}
\newtheorem{cor}[thm]{Corollary}
\newtheorem{lemma}[thm]{Lemma}

\newtheorem{defn}[thm]{Definition}

\newtheorem{rmk}[thm]{Remark}
\newtheorem*{crd1}{Coordinates I}
\newtheorem*{crd2}{Coordinates II}
\newtheorem*{crd3}{Coordinates III}
\newtheorem*{crd4}{Coordinates IV}

\newcommand{\Sym}{\mathrm{Sym}}
\newcommand{\OO}{\mathcal{O}}
\newcommand{\Bl}{\mathrm{Bl}}
\newcommand{\GL}{\text{GL}}
\newcommand{\codim}{\mathrm{codim}}

\makeatletter
\newcommand{\xdashrightarrow}[2][]{\ext@arrow 0359\rightarrowfill@@{#1}{#2}}
\newcommand{\xdashleftarrow}[2][]{\ext@arrow 3095\leftarrowfill@@{#1}{#2}}
\newcommand{\xdashleftrightarrow}[2][]{\ext@arrow 3359\leftrightarrowfill@@{#1}{#2}}
\def\rightarrowfill@@{\arrowfill@@\relax\relbar\rightarrow}
\def\leftarrowfill@@{\arrowfill@@\leftarrow\relbar\relax}
\def\leftrightarrowfill@@{\arrowfill@@\leftarrow\relbar\rightarrow}
\def\arrowfill@@#1#2#3#4{%
  $\m@th\thickmuskip0mu\medmuskip\thickmuskip\thinmuskip\thickmuskip
   \relax#4#1
   \xleaders\hbox{$#4#2$}\hfill
   #3$%
}
\makeatother

\newcommand{\sqp}[1][]{%
\ifthenelse{\isempty{#1}}{\mathcal{P}(\square_2)}{\mathcal{P}^{#1}(\square_2)}%
}

\newcommand{\PP}{{\mathbb{P}}}

\DeclareMathOperator\PGL{PGL}
\DeclareMathOperator\Hu{Hu}
\DeclareMathOperator\Gr{Gr}

\newcommand{\rleft}{\mathopen{}\mathclose\bgroup\left}
\newcommand{\rright}{\aftergroup\egroup\right}

\def\coloneqq{\mathrel{\mathop:}=}

\usepackage[bbgreekl]{mathbbol}
\usepackage{amsfonts}

\DeclareSymbolFontAlphabet{\mathbb}{AMSb}
\DeclareSymbolFontAlphabet{\mathbbl}{bbold}

\title[\resizebox{4.5in}{!}{The enumerative geometry of cubic hypersurfaces: point and line conditions}] {The enumerative geometry of cubic hypersurfaces: point and line conditions}

\author[M.\,Belotti]{Mara Belotti}
\address[M.\,Belotti]{Technische Universit\"at Berlin, Chair of Discrete Mathematics/Geometry, Stra{\ss}e des 17. Juni 136, 10623 Berlin, Germany}
\email{belotti@math.tu-berlin.de}
\author[A.\,Danelon]{Alessandro Danelon}
\address[A.\,Danelon]{Eindhoven University of Technology, Department of Mathematics and Computer Science, Groene Loper, MetaForum Building, Eindhoven, The Netherlands}
\email{a.danelon@tue.nl}
\author[C.\,Fevola]{Claudia Fevola}
\address[C.\,Fevola]{Max Planck Institute for Mathematics in the Sciences\\ Inselstraße 22\\ 04103 Leipzig\\Germany}
\email{claudia.fevola@mis.mpg.de}

\author[A.\,Kretschmer]{Andreas Kretschmer}
\address[A.\,Kretschmer]{Fakult\"at f\"ur Mathematik\\Institut f\"ur Algebra und Geometrie\\Otto-von-Guericke-Universit\"at Magdeburg\\Universit\"atsplatz 2\\ 39106 Magdeburg\\Germany}
\curraddr{}
\email{andreas.kretschmer@ovgu.de}
\thanks{}

\pgfplotsset{compat=1.17}

\begin{document}
\nocite{*}

\begin{abstract}
    In order to count the number of smooth cubic hypersurfaces tangent to a prescribed number of lines and passing through a given number of points, we construct a compactification of their moduli space. We term the latter a $1$--\textit{complete variety of cubic hypersurfaces} in analogy to the space of complete quadrics. Paolo Aluffi explored the case of plane cubic curves. Starting from his work, we construct such a space in arbitrary dimension by a sequence of five blow-ups. The counting problem is then reduced to the computation of five Chern classes, climbing the sequence of blow-ups. Computing the last of these is difficult due to the fact that the vector bundle is not given explicitly. Identifying a restriction of this vector bundle, we arrive at the desired numbers in the case of cubic surfaces.
\end{abstract}

\let\thefootnote\relax
\footnotetext{\hspace*{-14pt}
\begin{minipage}{353pt} MB has been funded by the Deutsche Forschungsgemeinschaft (DFG, German Research Foundation) under Germany's Excellence Strategy – The Berlin Mathematics Research Center MATH$^+$ (EXC-2046/1, project ID 390685689).
AD is supported by Jan Draisma's Vici grant 639.033.514 from
the NWO, Netherlands Organisation for scientific research,
\begin{minipage}{.019\textwidth}
\includegraphics[width=\textwidth]{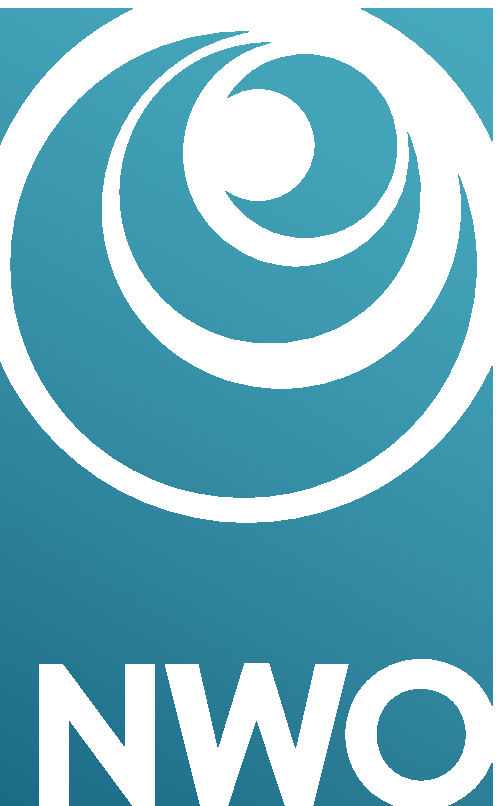}
\end{minipage}. AK is supported by the Deutsche Forschungsgemeinschaft (DFG, German Research Foundation) -- 314838170, GRK~2297 MathCoRe.
\end{minipage}}

\maketitle{}

\tableofcontents

\section*{Introduction}\label{introduction}
A famous moduli space in enumerative geometry is the \textit{space of complete quadrics}. This is a compactification of the space of smooth quadric hypersurfaces in $\PP(W)=\PP^n$, where $W$ is an $(n+1)$-dimensional vector space over an algebraically closed field $\mathbbl{k}$.
To construct this space, one starts with $V_0=\PP(\Sym^2(W))$ and considers the sequence of $n$ blow-ups obtained by iteratively blowing up the proper transforms of the loci of matrices with rank at most $i$.
For more details, we refer to \cite{manivel2020complete} and the references therein.
This variety has been used to answer the degree $2$ case of questions like:
\begin{center}
\textit{How many smooth degree $d$ hypersurfaces in $\PP^n$ are tangent to $\binom{n+d}{n}-1$ general linear spaces of various dimensions?}
\end{center}
The solutions to these problems are classically called \textit{characteristic numbers}.
In the case of quadrics, this question was first answered by Schubert \cite{schubert1879kalkul}, back in $1879$. What might sound like a rather basic question was later translated into a problem about the Chow ring of the space of complete quadrics, where beautiful results were achieved \cite{de1985complete,semple1948complete,vainsencher1982schubert}. More recently, the space of complete quadrics has proved useful to study some classical problems in algebraic statistics related to maximum likelihood estimation \cite{manivel2020complete,michalek2021maximum}. For quadrics, we know the characteristic numbers and also a space where to translate the question above into a cohomological problem.

Much less is known when it comes to higher degree hypersurfaces. To our knowledge the only case where all characteristic numbers are known are plane cubics and plane quartics. For the latter partial results were achieved in \cite{aluffi1991some,van1991characteristic}, and later a full description was given in \cite{vakil1999characteristic}. When it comes to numbers for cubic plane curves we have to go back to around $150$ years ago, when in the early 1870's Maillard 
\cite{maillard1871recherches} and Zeuthen \cite{zeuthen1873almindelige}
claimed to have computed them. Unfortunately, their methods were relying on assumptions that were not rigorously justified. 
It took more than a century to prove these numbers by using rigorous theoretical foundations provided by Fulton-MacPherson intersection theory, as in the works of Kleiman and Speiser \cite{kleiman1984enumerative}, and Aluffi \cite{Aluffi:1990,aluffi1991enumerative}. 
A particularly interesting feature of \cite{Aluffi:1990} is that in order to compute characteristic numbers, the author constructs a \textit{space of complete plane cubics}, which turns out to be the right compactification of the moduli space of plane cubics where to answer the above enumerative question. 
In a similar fashion as for complete quadrics, the space of complete plane cubics is constructed through a sequence of blow-ups.

As far as we know, the case of higher dimensional cubic hypersurfaces has been unexplored.
Our aim in this paper is to generalize the space of complete plane cubics in \cite{Aluffi:1990} and to construct what we call a space of $1$--\textit{complete cubic hypersurfaces}, which is the right space where to answer the following enumerative geometry question:
\begin{center}
\textit{What is the number of smooth cubic hypersurfaces in $\PP^n$ passing through $n_p$ general points and tangent to $\binom{n+3}{n}-n_p-1$ general lines?}
\end{center}
The paper is based on \cite{Aluffi:1990}, whose construction we find out to be generalizable to our specific setting.
We now focus on our construction and on the reason for a 1-complete variety of cubic hypersurfaces to be the right space where to answer the above question.

The moduli space of cubic hypersurfaces in $\PP(W)$ is naturally the projective space $\PP(\Sym^3(W))$ of dimension $\binom{n+3}{3}-1$.
We call \textit{line condition} the hypersurface in $\PP(\Sym^3(W))$ of cubics tangent to a given line in $\PP(W)$, and \textit{point condition} the hyperplane of cubics which contain a given point in $\PP(W)$.
We want to count the finite number of smooth cubics in the intersection of $n_p$ general point conditions and $\binom{n+3}{n}-n_p-1$ general line conditions. 
However, the intersection of such hypersurfaces might not be generically transverse.
A central role is indeed played  by the locus where all line conditions intersect, and this turns out to be the set of non-reduced cubics of the form $\lambda \mu^2$ for linear forms $\lambda$ and $\mu$.
This description of the \textit{base locus} is indeed the reason of our focus on \textit{lines} rather than more general linear spaces.
The goal of our construction is to obtain a variety birational to $\PP(\Sym^3(W))$ but such that in this new space, the proper transforms of the line conditions do no longer intersect.  We call such a variety a $1$-complete variety of cubic hypersurfaces. It turns out that as in \cite{Aluffi:1990}, it is enough to blow-up five times along irreducible components of the loci where the proper transforms of the line conditions intersect.
Our ultimate goal is then to compute the correction term that is needed to be subtracted from the bound provided by Bézout's theorem. This number can be expressed in terms of certain Chern classes related to the behavior of the cubic hypersurfaces in the blow-up process.

This paper is organized as follows. In Section~\ref{sec:1} we give the definition of a $1$-complete variety of cubic hypersurfaces $\tilde{V}$ and Theorem~\ref{thm:counting} proves that the intersection numbers we will compute in this variety coincide with the characteristic numbers we were aiming for.

Section~\ref{sec:2} concerns the construction of the $1$-complete variety $\tilde{V}$ achieved by performing five blow-ups. In each subsection we spell out the details of each blow-up by expressing its equations, the support of the intersection of the proper transforms of the line conditions, and the equations for this intersection.
This intersection is then taken to be the center of the next blow-up.
The construction ends with Corollary~\ref{cor:empty} where we show that the proper transforms of the line conditions no longer intersect.

Section~\ref{sec:3} is devoted to the Chow rings of the five centers defined in the previous section and to the computation of the intersection classes needed for the correction term. 

In the final Section~\ref{sec:4} we gather the data computed so far and provide the characteristic numbers for cubic surfaces in projective $3$-space. The proof of Theorem~\ref{thm:B4}\eqref{item:Chern_classes_E} and Remark~\ref{rmk:goeswrong} explain what is missing to determine the characteristic numbers with respect to line conditions in higher dimensions.

The code used in this work together with computational results is available~at
\[
	\texttt{\href{https://mathrepo.mis.mpg.de/CountingCubicHypersurfaces}{https://mathrepo.mis.mpg.de/CountingCubicHypersurfaces}}.
\]
\textbf{Acknowledgments.}
The authors wish to thank Mateusz Micha\l{}ek for presenting the problem to us and for the precious help offered throughout the way. Special thanks go to Paolo Aluffi for his hints and the time he dedicated to answer our long emails about his thesis. We also thank Tim Seynnaeve and Fulvio Gesmundo for valuable discussions. Finally, we are grateful to the organizers of the online workshop \href{https://sites.google.com/view/react-2021}{REACT}, which gave us the opportunity to meet and start this project.

\section{First associated hypersurfaces and the Hurwitz map}\label{sec:1}
    We fix an integer $d\geq 2$, an algebraically closed field $\mathbbl{k}$ of characteristic $0$ or strictly greater than $d$, and a $\mathbbl{k}$-vector space $W$ of dimension $n+1$ with $n \geq 2$.
	We refer to \cite[Section~3.2.E]{Gelfand1994Discriminants} for the notion of higher associated hypersurfaces of a projective variety.
	Specifically, we are interested in the following case: Let $X \coloneqq \mathcal{V}(f) \subseteq \PP(W)$ be an irreducible projective hypersurface of degree $d\geq 2$, defined by an irreducible homogeneous polynomial $f \in \mathbbl{k}[x_0, \ldots, x_n]_d$. If $X$ is smooth, its \emph{first associated hypersurface} $\mathcal{Z}_1(X) \subseteq \Gr(2,W)$ consists of all lines $\ell \subseteq \PP(W)$ such that $\ell$ is tangent to $X$ at some point or, more precisely, $\dim(\ell\cap E T_xX)=1$ for some point $x\in\ell\cap X$, where $E T_xX$ is the embedded tangent space of $X$ at a point $x$. If instead $X$ is singular, we first consider the lines $\ell$ for which there exists a smooth point satisfying the above conditions and then take the Zariski closure of this set in the Grassmannian $\Gr(2,W)$.

	In \cite[Proposition~2.11]{Gelfand1994Discriminants} it is shown that $\mathcal{Z}_1(X)$ is an irreducible hypersurface in $\Gr(2,W)$. Moreover, if $X$ is smooth, \cite[Theorem~1.1]{Sturmfels2017Hurwitz} shows that $\mathcal{Z}_1(X)$ is defined by an irreducible element $\Hu_f$ of degree $d(d-1)$ in the projective coordinate ring of $\Gr(2,W)$, called the \emph{Hurwitz form}, written as a degree $d(d-1)$ homogeneous polynomial in the Plücker coordinates, uniquely only up to the degree $d(d-1)$ piece of the ideal generated by the Plücker relations.
	On the open set of $\PP(\Sym^d(W))$ parametrizing smooth degree $d$ hypersurfaces, we can define an injective morphism sending $X$ to the degree $d(d-1)$ hypersurface $\mathcal{Z}_1(X)$ of $\Gr(2,W)$. The set of these hypersurfaces in $\Gr(2,W)$ is parametrized by the projective space $\PP(|\mathcal{O}_{\Gr(2,W)}(d(d-1))|)$.
	We define the \textit{Hurwitz map} to be the rational map:
    \begin{equation*}
        \begin{tikzcd}
	        \Hu: \PP(\Sym^d(W)) \arrow[r, dashed] & \PP(|\mathcal{O}_{\Gr(2,W)}(d(d-1))|),\quad
	        {[f]} \arrow[r, mapsto] & {[\Hu_f]}.
        \end{tikzcd}
    \end{equation*}
	For instance, if $n=2$, this map is simply the one taking a degree $d$ plane curve into its dual curve. Following \cite[Example~2.2]{Sturmfels2017Hurwitz}, $\Hu_f$ can be computed 
	as the resultant of the homogeneous polynomials of degree $d-1$ in the variables $s$ and $t$ given by the two partial derivatives of $f(s v_0 + t w_0, \ldots, s v_n + t w_n)$.
	As a polynomial in $s$ and $t$, the coefficients of the latter are bihomogeneous of degree $(1,d)$ in the coefficients of $f$ and the variables $v_i, w_i$, respectively.
	It follows that the polynomial $\Hu_f$ is bihomogeneous of degree $(2(d-1), 2d(d-1))$ with respect to the aforementioned variables.
	By \cite[Example~2.2]{Sturmfels2017Hurwitz}, $\Hu_f$ can even be expressed as a polynomial in the Pl\"ucker coordinates $p_{0,1},p_{0,2},\dots,p_{n,n+1}$ of the Grassmannian $\Gr(2,W)$ given by the $2 \times 2$ minors $p_{ij} = v_i w_j - v_j w_i$. Hence, $\Hu_f$ is bihomogeneous of degree $(2(d-1),d(d-1))$ in the coefficients of $f$ and the Pl\"ucker coordinates, respectively.
	Notice that the polynomial obtained in this way makes sense also for non-smooth, reducible and even non-reduced hypersurfaces $\mathcal{V}(f)$.
	
	\begin{rmk}
	    The rational map induced by the linear system generated by all line conditions in $H^0(\PP(\Sym^d(W)),\mathcal{O}(2(d-1)))$ is closely related to $\Hu$. Composing the former with a suitable linear embedding into $\PP(|\mathcal{O}_{\Gr(2,W)}(d(d-1))|)$ gives the latter.
	\end{rmk}
	
	In the same line of \cite{Aluffi:1990}, we define the \textit{point condition} $P^p$ and the \textit{line condition} $L^{\ell}$ as the hypersurfaces in $\PP(\Sym^d(W))$ consisting  of the degree $d$ hypersurfaces, respectively, containing the point $p$, and tangent to the line $\ell$.
	
    \begin{lemma}
        The indeterminacy locus of the rational Hurwitz map $\Hu$ is precisely the intersection of all line conditions, which in turn set-theoretically agrees with the subset $S_0 \subseteq \PP(\Sym^d(W))$ of the hypersurfaces defined by degree $d$ homogeneous polynomials divisible by the square of some non-constant polynomial.
    \end{lemma}
    
    \begin{proof}
    Fixing a line $\ell \in \Gr(2,W)$, the polynomial $\Hu_f(\ell)$ is a homogeneous degree $2(d-1)$ polynomial in the coefficients of $f$. Its vanishing set agrees with $L^{\ell} \subseteq \PP(\Sym^d(W))$, hence for the first claim it is enough to see that $\Hu_f(\ell)$, for fixed $\ell$, is irreducible as a polynomial in the coefficients of $f$. This is clearly a property invariant under the action of $\PGL_n$, so we can consider the line $\ell = \langle e_0, e_1 \rangle$. Then $\Hu_f(\ell)$ is precisely the discriminant of the generic homogeneous degree $d$ polynomial in two variables $x_0, x_1$, and this is indeed known to be an irreducible polynomial of degree $2(d-1)$ if $\mathrm{char}(\mathbbl{k}) \neq 2$.\\
    The indeterminacy locus of $\Hu$ is the set of $[f]$ such that $\Hu_f(\ell) = 0$ for every line~$\ell$. In this case, the singular locus of the closed subscheme $\mathcal{V}(f) \subseteq \PP(W)$ must have codimension~$0$, otherwise the general line would intersect $\mathcal{V}(f)$ transversally in $d$ distinct smooth points. But the singular locus of $\mathcal{V}(f)$ can only have codimension~$0$ if $\mathcal{V}(f)$ has a non-reduced component, so $f$ is divisible by the square of some non-constant polynomial.
	\end{proof}
	
    This allows us to present the following definition.
    
	\begin{defn}
		A \textit{$1$-complete variety of degree $d$ hypersurfaces} is a morphism $\pi\colon \tilde{V} \rightarrow \PP(\Sym^d(W))$ from a smooth projective variety $\tilde{V}$ which is an isomorphism outside $\pi^{-1}(S_0)$ resolving $\Hu$, i.e., such that the proper transforms of all line conditions $L^{\ell}$ in $\tilde{V}$ do not intersect:
		\begin{equation*}
			\begin{tikzcd}
			\tilde{V} \ar[d, "\pi" left] \ar[dr, "\widetilde{\Hu}" above right]\\
			\PP(\Sym^d(W))\ar[r, dashed, "\Hu"']
			& \PP(|\mathcal{O}_{\Gr(2,W)}(d(d-1))|). \\
			\end{tikzcd}
		\end{equation*}
	\end{defn}
	
	An analogous construction for tangency with respect to $s$-dimensional planes instead of lines would lead to the definition of \textit{$s$-complete varieties of degree $d$ hypersurfaces}. For $s \geq 2$, however, the intersection of all $s$-plane conditions set-theoretically agrees with the subset of $\PP(\Sym^d(W))$ given by all degree $d$ hypersurfaces with singular locus of dimension $\geq n-s$, which to our knowledge is not as easily parametrizable as the set $S_0$ of non-reduced hypersurfaces. 
	
	\begin{thm}\label{thm:counting}
		We write $V_0 \coloneqq \PP(\Sym^d(W))$. Let $\tilde{V}$ be a $1$-complete variety of degree $d$ hypersurfaces as above and let $F \subseteq V_0 \setminus S_0$ be an irreducible, locally closed subset. Denote by $\tilde{F} \subseteq \tilde{V}$ the proper transform of the closure $\overline{F}$ and by $\tilde{L}^{\ell}, \tilde{P}^p \subseteq \tilde{V}$ the line and point conditions of $\tilde{V}$, i.e., the proper transforms in $\tilde{V}$ of the irreducible hypersurfaces $L^{\ell}, P^p \subseteq V_0$ corresponding to line and point conditions of $V_0$, respectively, for the line $\ell \subseteq \PP(W)$ and the point $p \in \PP(W)$.
		\begin{enumerate}
			\item\label{item:proper_intersection} For any finite set of subvarieties $A_1, \ldots, A_r \subseteq \tilde{V}$, there exist a point $p$ and a line $\ell$ such that $\tilde{P}^p$ and $\tilde{L}^{\ell}$ both intersect every $A_i$ properly, i.e. in the expected dimension. In fact, this is the case for the general point and the general line.
			\item\label{item:fin_pts} If $r = \dim(F)$, there exist $r$ lines $\ell_1, \ldots,\ell_r$ such that the corresponding line conditions in $\tilde{V}$ intersect $\tilde{F}$ in finitely many points, mapping to $F$ under~$\pi$. Again, this is the case for general lines $\ell_1, \ldots, \ell_r$.
			\item\label{item:number_elements} The number of elements of $F$, counted with multiplicity, passing through $n_p$ general points and tangent to $n_{\ell}$ general lines such that $n_p + n_{\ell} = \dim(V_0) = \binom{n+d}{d}-1$ equals the degree of the $0$-cycle $\tilde{P}^{n_p} \cdot \tilde{L}^{n_{\ell}} \cdot \tilde{F} \in \mathrm{CH}_0(\tilde{V})$, where $\tilde{P}, \tilde{L}$ denote the cycle classes of any point and line condition $\tilde{P}^p, \tilde{L}^{\ell}$ in $\tilde{V}$.
			\item\label{item:multiplicity} If the points and lines are general enough, the multiplicity of every element of $F$ from \eqref{item:number_elements} is $1$.
		\end{enumerate}
	\end{thm}
	
	\begin{proof}
		Both \eqref{item:fin_pts} and \eqref{item:number_elements} follow from \eqref{item:proper_intersection}, the arguments being the same as in \cite[Proposition~1]{Aluffi:1990}. For \eqref{item:proper_intersection}, we also mimic the strategy of \cite[Proposition~1]{Aluffi:1990} and assume the conclusion is false, i.e., there are subvarieties $A_1, \ldots, A_r$ such that for every line $\ell$, the line condition $\tilde{L}^{\ell}$ intersects non-properly at least one of the $A_i$. As $\tilde{L}^{\ell}$ is an irreducible hypersurface and all $A_i$ are irreducible, this means that every line condition in $\tilde{V}$ contains at least one of the $A_i$. Pick a point $p_i \in \widetilde{\Hu}(A_i)$ for each $i = 1, \ldots, r$ and denote by $G^{\ell} \subseteq \PP(|\mathcal{O}_{\Gr(2,W)}(d(d-1))|)$ the hyperplane given by all degree $d(d-1)$ elements of the projective coordinate ring of $\Gr(2,W)$  vanishing at $\ell$. Then the coefficients of the linear equation defining $G^{\ell}$ are those of the polynomial $(\sum_{ij} l_{ij} p_{ij})^{d(d-1)}$, where $l_{ij}$ is the $ij$-th Plücker coordinate of the line $\ell$. It follows that $\Hu(L^{\ell} \setminus S_0) \subseteq G^{\ell}$ and therefore $\widetilde{\Hu}(\tilde{L}^{\ell}) \subseteq G^{\ell}$. In particular, every hyperplane $G^{\ell}$ contains at least one of the finitely many points $p_i$. Dually, in $\check{\PP}(|\mathcal{O}_{\Gr(2,W)}(d(d-1))|)$ this means that all points corresponding to the hyperplanes $G^{\ell}$ are contained in the finite union of hyperplanes corresponding to the points $p_i$. However, the set of points corresponding to the $G^{\ell}$ is the image of the $d(d-1)$-Veronese embedding $\Gr(2,W) \rightarrow \check{\PP}(|\mathcal{O}_{\Gr(2,W)}(d(d-1))|)$ and thus irreducible. Hence, this image would have to be contained in a single hyperplane. In other words, switching back to the primal setting, there exists some $p_i$ that is contained in all hyperplanes $G^{\ell}$. We can assume this point to be $p_1$. Then $p_1$ corresponds to a non-zero element in the degree $d(d-1)$ part of the projective coordinate ring of $\Gr(2,W)$ that, as a polynomial in the Plücker coordinates, must vanish at all lines in $\PP(W)$, hence on all of  $\Gr(2,W)$. This, of course, is impossible.
		
		For \eqref{item:multiplicity}, the arguments of \cite[Lemma~2 and Theorem~I(2)]{Aluffi:1990} generalize to our setting if one substitutes terms like `finitely many' in Lemma~2 with `in codimension~2.'
	\end{proof}

	\section{A $1$-complete variety of cubic hypersurfaces}\label{sec:2}
	
	This section is dedicated to our construction of a $1$-complete variety of cubic hypersurfaces. We start from the projective space $V_0 \coloneqq \PP(\Sym^3 (W))$ parametrizing cubic hypersurfaces in $\PP^n=\PP(W)$ and blow up five times along smooth centers. At each level, these are given by an irreducible component of the intersection of all proper transforms of the line conditions. We will also refer to cubic hypersurfaces as cubics.
	
	We saw in the previous section that $S_0$ coincides with the set of cubic hypersurfaces divisible by the square of a non-constant polynomial.
	Hence, $S_0$ is the image of the morphism
	\begin{equation}\label{eq:phi0}
    	\begin{tikzcd}
    	\phi_0:\PP^n\times\PP^n \arrow[r] & \PP(\Sym^3(W)), \quad ([\lambda],[\mu]) \arrow[r, mapsto] & {[\lambda\mu^2]}. 
    	\end{tikzcd}
	\end{equation}
	As $\phi_0$ is injective, $S_0$ is a subvariety of $\PP(\Sym^3(W))$ of dimension $2n$.
	Let $\Delta$ denote the diagonal in $\PP^n \times \PP^n$. We write $B_0$ for the locus $\phi_0(\Delta)$ of triple hyperplanes.\\
	The following result is a direct generalization of \cite[Lemma 0.1]{Aluffi:1990}. We will often use it without explicit reference. The statement does not depend on the choice of the line $\ell\subseteq \PP^n$, hence can be verified on the equation of any line condition.  
	\begin{lemma}\label{lem:lemma0.1}
		Let $L^{\ell}$ be the line condition in $\PP(\Sym^3(W))$ corresponding to $\ell\subseteq \PP^n$. Then:
		\begin{enumerate}
			\item If $c\in L^{\ell}$, then $L^{\ell}$ is smooth at $c$ if and only if $c$ intersects $\ell$ with multiplicity exactly 2 at a point. In particular, the line conditions are generically smooth along the locus $S_0$ of non-reduced cubics.
			\item If $c$ intersects $\ell$ with multiplicity 3 at a point, then $L^\ell$ has multiplicity 2 at $c$. In particular, the line-conditions have multiplicity 2 along the locus $B_0$ of triple lines.
			\item The tangent hyperplane to $L^{\ell}$ at a smooth point $c$ consists of the cubics containing the point of tangency of $c$ to $\ell$. The tangent cone in $V_0$ to $L^{\ell}$ at a cubic $c$ intersecting $\ell$ in a triple point $p$ is supported on the hyperplane in $V_0$ consisting of the cubics containing $p$.
		\end{enumerate}
	\end{lemma}
	An immediate generalization of Lemma 0.2 in \cite{Aluffi:1990} implies that the map $\phi_0$ is an isomorphism when restricted to $\PP^n\times\PP^n\setminus\Delta$. 

    The next subsections explain in details the construction of the 1-complete variety of cubic hypersurfaces $\tilde{V}$, a schematic view over this construction and the notation employed is englobed in the following diagram:

	\begin{equation*}
	\begin{tikzcd}\label{construction}
	&  V_5 \ar[d, "\pi_5"]\\
	&  V_4 \ar[d, "\pi_4"]
	& B_4=\PP(\mathcal{E}) \ar[d] \ar[l, "j_4"]\\
	&  V_3 \ar[d, "\pi_3"]
	& B_3 = S_3 \ar[d] \ar[l, "j_3"]
	& \Bl_{\Delta} \PP^n \times \PP^n \ar[l, "\phi_3"] \ar[d, equal]\\
	B_2 \ar[r,"j_2"] \ar[d]
	&  V_2 \ar[d, "\pi_2"]
	& S_2 \ar[d] \ar[l]
	& \Bl_{\Delta} \PP^n \times \PP^n \ar[l,"\phi_2"] \ar[d, equal]\\
	B_1 \ar[r, "j_1"] \ar[d]
	& V_1 \ar[d, "\pi_1"]
	& S_1 \ar[l] \ar[d]
	& \Bl_{\Delta} \PP^n \times \PP^n  \ar[l, "\phi_1"] \ar[d]\\
	B_0 = \nu_3(\PP^n) \ar[r,"j_0"] 
	& V_0 = \PP(\Sym^3(W))
	& S_0  \ar[l]
	& \PP^n \times \PP^n \ar[l, "\phi_0"]
	\end{tikzcd}
	\end{equation*}
	The center of each blow-up is denoted by $B_i$, while the blow-ups are called $V_i$. Hence $V_{i+1}$ denotes the blow-up of $V_{i}$ at the center $B_i$ and $\pi_{i+1}:V_{i+1}\to V_i$ the corresponding blow-up map, for $i=0,\dots,4$. For $i\leq3$, $S_i$ indicates the proper transform of the locus $S_0$ in $V_i$. 
	
	We note that the above diagram is analogous to the one in \cite[p.~514]{Aluffi:1990}. Our construction of the $1$-complete variety of cubic hypersurfaces $V_5$ is indeed a direct generalization to higher dimension of the one performed by Aluffi for plane cubic curves. The same number of blow-ups is needed to empty the locus where the proper transforms of the line conditions intersect. One crucial difference, however, is that the center $B_4$ of the fifth blow-up in our case is the projectivization of a vector bundle $\mathcal{E}$ of rank stricly higher than $1$ which is not a priori known explicitly. Its identification and the computation of its Chern classes are difficult tasks and have no analog in \cite{Aluffi:1990}. In particular, Proposition~\ref{prop:identify_E} provides a new proof of the important Lemma~4.2 of \cite{Aluffi:1990}.

	\setcounter{subsection}{-1}
	\subsection{Space of cubic hypersurfaces }\label{subsec:V0}
	In what follows let $\mathbbl{k}$ be an algebraically closed field of characteristic $\neq 2, 3$ and $W$ a $\mathbbl{k}$-vector space of dimension $n+1$ with basis $e_0, \ldots, e_n$.
	Let us introduce some notation in order to develop the first blow-up. 
	Denote by $[a_I] = [a_{(0,0,0)} : \dots : a_{(n,n,n)}]$ the vector of $\binom{n+3}{3}$ projective coordinates for $V_0$, more explicitly each $a_{(i,j,k)}$ corresponds to the coefficient of the monomial $x_ix_jx_k$ in the equation for the associate cubic in $\PP^n$ where we assume $i\leq j \leq k$. We denote by $[n]$ the set of natural numbers between 0 and $n$. Then in the affine chart $D(a_{(0,0,0)})$, the ideal $\mathcal{I}(B_0)$ in $V_0$ determining the locus of triple hyperplanes is generated by the polynomials $f_J$, where $J$ denotes all multi-indices $(i,j,k)\in [n]^3$ with $i \leq j \leq k$ such that $i$ and $j$ are not both equal to zero, we have:
	\begin{alignat}{2}\label{eq:equations_B0}
	f_{(0,i,i)} &\coloneqq  3a_{(0,i,i)}-a_{(0,0,i)}^2 && \ \qquad\text{ for } i>0, \nonumber \\
	f_{(0,i,j)} &\coloneqq  3a_{(0,i,j)}-2a_{(0,0,i)}a_{(0,0,j)} && \ \qquad\text{ for } j>i>0, \nonumber\\
	f_{(i,i,i)} &\coloneqq  9a_{(i,i,i)}-a_{(0,0,i)}a_{(0,i,i)} && \ \qquad\text{ for } i>0,\\
	f_{(i,i,j)} &\coloneqq  3a_{(i,i,j)}-a_{(0,i,i)}a_{(0,0,j)} && \ \qquad \text{ for } i,j>0, i \neq j, \nonumber\\
	f_{(i,j,k)} &\coloneqq  3a_{(i,j,k)}-a_{(0,i,j)}a_{(0,0,k)} && \ \qquad\text{ for } k>j>i>0, \nonumber
	\end{alignat}
    these will provide equations for the center of the first blow-up. Note that $B_0$ is a smooth complete intersection of codimension $ \binom{n+3}{3}-1 - n$ inside this open chart. In what follows, when we write the affine coordinates $(a_I)$ we always assume the index $(0,0,0)$ to be excluded.
    \begin{rmk}\label{rmk:PGL}
    A line condition is a degree 4 hypersurface in $V_0$.
    Indeed, fix the line $\ell=\mathcal{V}(x_2, \ldots, x_n)\subseteq \PP^n$.
    The tangency condition for a cubic to such line is given by the vanishing of the resultant of its derivatives with respect to $x_0$ and $x_1$.
    Then the equation for the line condition $L^{\ell}$ in $D(a_{(0,0,0)}) \subseteq V_0$ is given by
    \begin{equation}\small{\label{eq:lineCondition}
		a_{(0,0,1)}^2a_{(0,1,1)}^2+18a_{(0,0,1)}a_{(0,1,1)}a_{(1,1,1)}-4a_{(0,1,1)}^3-4a_{(0,0,1)}^3a_{(1,1,1)}-27 a_{(1,1,1)}^2=0.}
    \end{equation}
    Using the action of $\PGL_n$, we can recover the equation for $L^{\ell}$ for any line ${\ell} \subseteq \PP^n$.
	\end{rmk}
    
	\subsection{First Blow-up}\label{subsec:oneblow}
		 Denote by $V_1$ the blow-up of the space $V_0$ along the center $B_0$, and $L_1$ the proper transform in $V_1$ of a line condition $L$.
		 
	\begin{crd1}\label{eq:equations_V1}
		Let $([a_I],[b_J])$ denote the projective coordinates on $V_0 \times \PP^{r-1}$, where $r$ is the codimension of $B_0$ as subvariety of $V_0$ and $J$ denotes all multi-indices $(i,j,k)\in[n]^3$ with $i \leq j \leq k$ such that $i$ and $j$ are not both zero. Then, by \cite[Exercise~17.14(b)]{EisenbudCommAlg} the blow-up $V_1$ is a closed subvariety of the affine chart $D(a_{(0,0,0)})$ given by the equations
		\begin{equation*}
		f_{J_1} b_{J_2} - f_{J_2} b_{J_1}=0,
		\end{equation*}
		where $J_1$, $J_2$ run over all multi-indices $J$ previously described and the $f_J$'s denote the equations in (\ref{eq:equations_B0}). \\
		We restrict to the affine chart $D(b_{(0,1,1)})$, where $V_1$ can be described by the affine coordinates $(a_{(0,0,1)}, \ldots, a_{(0,0,n)},a',b_J)$, where $J$ varies as above but we exclude $J = (0,1,1)$, and where the additional variable $a'$ corresponds to the equation $f_{(0,1,1)} = 3a_{(0,1,1)} - a_{(0,0,1)}^2$. The equations for $V_1$ in this affine open become 
		\begin{equation*}
		a' - f_{(0,1,1)}=0, \qquad f_J - b_J a'=0, \qquad \hbox{for all} \;\;  J \neq (0,1,1).
		\end{equation*}
		The equation for the exceptional divisor $E_1$ inside $V_1$ is then $a'=0$, and the $(b_J)$ provide coordinates in the chosen affine chart for the fiber of $E_1$ over a point in $B_0$. We will always exclude the index $J=(0,1,1)$ when considering the affine coordinates $(b_J)$.
	\end{crd1}
	
	\begin{lemma}\label{lemma:embedding_normal_bundles}
	Denote by $N_{\PP(W)} \PP(\Sym^d(W))$ the normal bundle for the $d$-th power or Veronese embedding $\PP(W) \hookrightarrow \PP(\Sym^d(W))$. Let $e \leq d$. Then there is a natural embedding of normal bundles
		\begin{equation*}
		    \alpha_{e,d}: N_{\PP(W)} \PP(\Sym^{e}(W)) \hookrightarrow N_{\PP(W)} \PP(\Sym^d(W)),
		\end{equation*}
	given by ``multiplication by $\lambda^{d-e}$'' in the fiber over $[\lambda] \in \PP(W)$.
	\end{lemma}
	
	\begin{proof}
		We write $R \coloneqq \mathbbl{k}[x_0, \ldots, x_n]$. The pullback of the Euler sequence on $\PP(\Sym^d(W))$ via $\nu_d$ is
		\begin{equation*}
		0 \rightarrow \OO_{\PP(W)} \overset{\nu_d^{\ast}(\varepsilon)}{\rightarrow} \Sym^d(W) \otimes \OO_{\PP(W)}(d) \rightarrow T\PP(\Sym^d(W))|_{\PP(W)} \rightarrow 0,
		\end{equation*}
		where $\nu_d^{\ast}(\varepsilon)$ is induced by the graded $R$-module homomorphism
		\begin{equation*}
		R \rightarrow \Sym^d(W) \otimes_{R} R(d), \ f \mapsto \sum_{|I|=d} \binom{d}{I} e_I \otimes (x^I f) = f \cdot (e_0 \otimes x_0 + \dots e_n \otimes x_n)^d,
		\end{equation*}
		The fiber of $\nu_{d}(\varepsilon)$ over $\lambda$ is therefore just multiplication by $\lambda^d = (\lambda_0 e_0 + \ldots + \lambda_n e_n)^d$. More generally, there is a commutative diagram with exact rows
		\begin{equation*}
		\begin{tikzcd}\label{Veronese}
		0 \ar[r]
		& \OO_{\PP(W)} \ar[r] \ar[d, equal]
		& W \otimes \OO_{\PP(W)}(1) \ar[r] \ar[d, "\alpha_{1,e}"]
		& T\PP(W) \ar[r] \ar[d, "\overline{\alpha_{1,e}} = \mathrm{d}\nu_{e}"]
		& 0 \\
		0 \ar[r]
		& \OO_{\PP(W)} \ar[r] \ar[d, equal]
		& \Sym^{e}(W) \otimes \OO_{\PP(W)}(e) \ar[r] \ar[d, "\alpha_{e,d}"]
		& T\PP(\Sym^{e}(W))|_{\PP(W)} \ar[r] \ar[d, "\overline{\alpha_{e,d}}"]
		& 0 \\
		0 \ar[r]
		& \OO_{\PP(W)} \ar[r]
		& \Sym^d(W) \otimes \OO_{\PP(W)}(d) \ar[r]
		& T\PP(\Sym^d(W))|_{\PP(W)} \ar[r]
		& 0.
		\end{tikzcd}
		\end{equation*}
		In here, $\alpha_{e,d}$ is induced by the graded $R$-module homomorphism which is multiplication by $(e_0 \otimes x_0 + \ldots + e_n \otimes x_n)^{d-e}$. It can be checked that $\overline{\alpha_{1,e}} = \mathrm{d}\nu_{e}$ is the differential of the $e$-th Veronese embedding. Then $\alpha_{e,d}$ induces the embedding of normal bundles we are looking for.
	\end{proof}
	
	For us, $e=2$, $d=3$. The exceptional divisor is $E_1 \cong \PP(N_{\PP(W)} \PP(\Sym^3(W)))$ and we call $B_1$ the image of $\PP(\alpha_{2,3})$ in $E_1$. The proper transform of $S_0$ in $V_1$ will be denoted by $S_1$.
	
	\begin{prop}
		The intersection of the proper transforms of all line conditions in $V_1$ is contained in the union $S_1 \cup B_1$.
	\end{prop}
	
	\begin{proof}
		It is enough to check that the intersection of the proper transforms of all line conditions and $E_1$ lies inside $B_1$. The intersection of the proper transform $L_1$ of a line condition $L$ with the fiber over $[\lambda^3] \in B_0$ is the image of the tangent cone of $L$ at the point $[\lambda^3]$ in the projectivized normal bundle $\PP(N_{B_0} V_0)$. By definition of $\alpha_{2,3}$ in Lemma~\ref{lemma:embedding_normal_bundles}, the fiber of $B_1$ over $[\lambda^3]$ consists of all cubics divisible by $\lambda$. Lemma~\ref{lem:lemma0.1}(iii) implies that the intersection of all tangent cones at $[\lambda^3]$ of all line conditions is contained in the set of cubics containing the hyperplane $\lambda$. This shows the claim.
	\end{proof}

	\begin{lemma}\label{lemma:liftofphi0}
		The universal property of blowing up gives a commutative diagram
		\[ 
		\begin{tikzcd}
		{\Bl_{\Delta}\PP^n\times\PP^n} \arrow[r, "\phi_1", "\cong"'] \arrow[d]
		& S_1 \arrow[d] \arrow[r, hook] & V_1 \arrow[d] \\
		{\PP^n\times\PP^n} \arrow[r, "\phi_0"']
		& S_0 \arrow[r, hook] & V_0.
		\end{tikzcd}
		\]
		Here, $\phi_1$ is an isomorphism, hence $S_1$ is smooth.
	\end{lemma}
	
	\begin{proof}
		We write $e$ for the exceptional divisor of $\Bl_{\Delta}\PP^n\times\PP^n$. The map $\phi_0$ lifts to a map  $\phi_1: \Bl_{\Delta}\PP^n\times\PP^n \to S_1$ via the universal property of blowing up. Indeed, it can be checked that the pullback of the ideal sheaf $\mathcal{I}(B_0)$ via $\phi_0$ is precisely the squared ideal sheaf $\mathcal{I}(\Delta)^2$ of the diagonal $\Delta \subseteq \PP^n \times \PP^n$, in particular the pullback of $\mathcal{I}(B_0)$ to $\Bl_{\Delta}\PP^n\times\PP^n$ is an effective Cartier divisor, as needed. Clearly, $\phi_1$ restricts to an isomorphism of $\Bl_{\Delta}\PP^n\times\PP^n \setminus e$ onto $S_1 \setminus E_1$. As $\Bl_{\Delta}\PP^n\times\PP^n$ and $S_1$ are projective varieties, $\phi_1$ is a closed map, so surjectivity follows. In order to prove the injectivity of $\phi_1$ we observe that $\phi_0$ is an injective morphism between varieties over an algebraically closed field, hence $\phi_0$ is universally injective. Base-changing $\phi_0$ along the blow-up map $\pi_1:V_1 \to V_0$ hence gives an injection $(\PP^n \times \PP^n) \times_{V_0} V_1 \to V_1$. The blow-up closure lemma ensures that $\Bl_{\Delta}\PP^n\times\PP^n$ is naturally a closed subscheme of $(\PP^n \times \PP^n) \times_{V_0} V_1$, and the composition $\Bl_{\Delta}\PP^n\times\PP^n \to V_1$ agrees with $\phi_1$, showing that $\phi_1$ is injective.
		By \cite[Corollary~14.10]{harris2013algebraic}, it remains to show that $(d\phi_1)_p: T_p (\Bl_{\Delta}\PP^n\times\PP^n) \to T_{\phi_1(p)} V_1$ is injective for all $p$ in the exceptional divisor $e$ of $\Bl_{\Delta}\PP^n\times\PP^n$. This matter is local and invariant under the $\PGL_n$-action, so we can assume $p$ to lie in the fiber of $([1:0:\cdots:0],[1:0:\cdots:0]) \in \Delta$. Choose local coordinates
		\[([1 : \lambda_1 : \cdots : \lambda_n], [1 : \mu_1 : \cdots : \mu_n]) \in \PP^n\times\PP^n.\]
		The equations for $\Delta$ are $u_i\coloneqq \lambda_i - \mu_i =0$ for all $i \in \{1,\dots ,n\}$. Thus, $\Bl_{\Delta}\PP^n\times\PP^n$ is described by the points $(\mu_1,\dots,\mu_n,u_1,\dots,u_n, [s_1,\dots,s_n])$  such that $u_i s_{j} - u_j s_{i}=0$ for all $i,j$. In the affine chart $D(s_1)$, the morphism $\phi_1$ is given explicitly in the affine coordinates $(\mu_1,\dots,\mu_n,u_1,s_2,\dots,s_n)$ by 
		\begin{equation*}
		\begin{split}
		a_{(0,0,1)}&=3\mu_1 + u_1,\\
		a_{(0,0,i)}&=3\mu_i + s_i u_1 \;\; \hbox{ for }i>1,\\
		a'&=-u_1^2,\\
		b_{(0,i,i)}&=s_i^2 \;\; \quad \qquad \;\,  \hbox{ for }i>1,\\
		b_{(1,i,i)}&=2\mu_i s_i \;\;  \qquad \hbox{ for }i>1,\\
		b_{(1,1,i)}&=2\mu_1 s_i \;\;\qquad  \hbox{ for }i>1,
		\end{split}\qquad \qquad 
		\begin{split}
		b_{(1,1,1)}&=2\mu_1,\\
		b_{(0,1,i)}&=2s_i \;\;\quad \;\;\;\qquad \;\;\quad \hbox{ for }i>1,\\
		b_{(0,i,j)}&=2s_i s_j \;\;\,\;\;\; \qquad \;\;\quad\hbox{ for }j>i>1,\\
		b_{(i,i,i)}&=2 \mu_i s_i^2 \;\; \;\;\;\qquad \;\;\quad\hbox{ for }i>1,\\
		b_{(i,i,j)}&=2\mu_i s_i s_j \;\;\qquad \;\;\quad \hbox{ for } i,j>1, i \neq j,\\
		b_{(i,j,k)}&=2s_k(\mu_i s_j+\mu_j s_i)\;\hbox{ for }k>j>i>0.
		\end{split}
		\end{equation*}
		The exceptional divisor $e$ now has equation $u_1 = 0$. This explicit description of $\phi_1$ allows us to conclude the proof by checking the non-degeneracy of the Jacobian at every point. Indeed, the $2n$ row vectors in the Jacobian corresponding to $a_{(0,0,i)}$ for $1 \leq i \leq n$, to $b_{(0,1,i)}$ for $2 \leq i \leq n$ and to $b_{(1,1,1)}$ are linearly independent.
	\end{proof}

	\begin{lemma}\label{lem:tangency} The set-theoretic intersection of $B_1$ and $S_1$ is $\phi_1(e)$. Moreover, the proper transforms of the line conditions are generically smooth and tangent to $E_1$ along $B_1$.
	\end{lemma}
	
	\begin{proof}
		Since $\phi_1$ is an isomorphism, we have $\phi_1(e) = S_1 \cap E_1$ and it suffices to show $\phi_1(e) \subseteq B_1$. By invariance under projective transformations, it suffices to see this for the fiber in $E_1$ over $[x_0^3] \in B_0$. Using the coordinates described above, the intersection of this fiber with $B_1$ in $V_1$ is described by the equations  $a_{(0,0,i)}=0$ for all $1 \leq i \leq n$, $a'=0$ and $b_J=0$ for every multi-index $J$ whose first entry is non-zero. The explicit description of $\phi_1$ shows that the image of the fiber of $([1:0:\cdots:0],[1:0:\cdots:0]) \in \PP^n \times \PP^n$ satisfies all these equations, proving the claim.\\
		For the second point, the invariance under the natural action of $\PGL_n$ on $V_1$ allows us to verify the claim for the line-condition corresponding to the line $\ell = \mathcal{V}(x_2, \ldots, x_n)$.
		We can restrict to the affine open $D(a_{(0,0,0)}) \cap D(b_{(0,1,1)})$ where we have local coordinates (see Remark~\ref{eq:equations_V1}). The equation for the line condition $L^{\ell}$ in $D(a_{(0,0,0)}) \subseteq V_0$ is given in (\ref{eq:lineCondition}).
		Plugging in $3a_{(0,1,1)}=a'+a_{(0,0,1)}^2$ and $27a_{(1,1,1)}=3b_{(1,1,1)}a'+a_{(0,0,1)}(a'+a_{(0,0,1)}^2)$, we get the equation 
		\[(a')^2(12b_{(1,1,1)}a_{(0,0,1)}-4a_{(0,0,1)}^2-4a'-9b_{(1,1,1)}^2)=0,\]
		which outside of $E_1$ describes the proper transform $L^{\ell}_1$ of the line condition in the chosen affine chart, whose equation is therefore $-4a'-(3b_{(1,1,1)}-2a_{(0,0,1)})^2=0$. Since the equation of $E_1$ in the local coordinates is $a'=0$, every point of $E_1$ belonging to the proper transform is indeed a tangency point. Moreover, the equation shows that the proper transform is smooth in this entire affine open.
	\end{proof}
	
	\begin{lemma}\label{lemma:equations_B1}
		The ideal of $B_1 \subseteq V_1$ in the open $D(a_{(0,0,0)})$ is generated by the equations
		\begin{alignat*}{2}
		f_J &= 0 &&\;\;\text{ for all } J, \\
		f_{(i,i,i)}' &:= 3b_{(i,i,i)} - 2a_{(0,0,i)} b_{(0,i,i)} = 0 &&\;\;\text{ for all } i>0, \\
		f_{(i,i,j)}' &:= 3b_{(i,i,j)} - a_{(0,0,i)} b_{(0,i,j)} = 0  &&\;\;\text{ for all } i,j>0, \ i \neq j, \\
	    f_{(i,j,k)}' &:= 3b_{(i,j,k)} - a_{(0,0,i)} b_{(0,j,k)} - a_{(0,0,j)} b_{(0,i,k)} = 0  &&\;\;\text{ for all } k>j>i>0.
		\end{alignat*}
	\end{lemma}
	
	These equations clearly form a regular sequence, so $B_1$ is a complete intersection in the open chart. In the affine chart  $D(b_{(0,1,1)})$, we can moreover replace the first set of conditions by $a'=0$, as above.
	
	\begin{proof}
		From the commutative diagram in the proof of Lemma~\ref{lemma:embedding_normal_bundles}, the fiber over $[\lambda^3] \in B_0$ of the normal bundle can be naturally identified with the vector space $\Sym^3(W)/\langle \lambda^2 x_0, \ldots, \lambda^2 x_n \rangle$. We want to understand how an element in the latter corresponds to an element of the fiber $E_1|_{[\lambda^3]}$ if explicitly written in the coordinates from the description of $V_1$ in (\ref{eq:equations_V1}). The answer is provided by the conormal sequence
		\begin{equation*}
		0 \rightarrow I/I^2 \rightarrow \Omega_{\mathbbl{k}[a_I]} \otimes_{\mathbbl{k}[a_I]} \mathbbl{k}[a_I]/I \rightarrow \Omega_{\mathbbl{k}[a_I]/I} \rightarrow 0.
		\end{equation*}
		Any point $k \in  \Sym^3(W)/\langle \lambda^2 x_0, \ldots, \lambda^2 x_n \rangle$ can be uniquely represented  as a cubic not containing the monomials $x_0^3, x_0^2 x_1, \ldots, x_0^2 x_n$. If we write $k = k_{(0,1,1)} x_0 x_1^2 + \ldots + k_{(n,n,n)} x_n^3$, then $k$ corresponds to the element $\sum_J y_J \overline{f_J}$ in  $\left( I/I^2 \otimes \mathbbl{k}[a_I]/\mathfrak{m}_{[\lambda^3]} \right)^{\vee}$ via
		\begin{alignat*}{2}
		b_{(i,j,k)} &= 3 k_{(i,j,k)} - a_{(0,0,k)} k_{(0,i,j)} &&\hspace{3cm} \qquad \text{ for all } k>j>i>0, \\
		b_{(i,i,j)} &= 3 k_{(i,i,j)} - a_{(0,0,j)} k_{(0,i,i)} &&\hspace{3cm} \qquad\text{ for all } i,j>0, i \neq j, \\
		b_{(i,i,i)} &= 9 k_{(i,i,i)} - a_{(0,0,i)} 
		k_{(0,i,i)} &&\hspace{3cm} \qquad\text{ for all } i>0, \\
		b_{(0,i,j)} &= 3 k_{(0,i,j)} &&\hspace{3cm} \qquad\text{ for all } j \geq i > 0.
		\end{alignat*}
		In the fiber over $[\lambda^3]\in B_0$ we have $a_{(0,0,i)} = 3\lambda_i$ and  it is easy to see that the cubic $k$ is divisible by $\lambda = x_0 + \lambda_1 x_1 + \ldots + \lambda_n x_n$ if and only if $k$ satisfies the equations
		\begin{alignat*}{2}
		3 k_{(i,j,k)} &= a_{(0,0,i)} k_{(0,j,k)} + a_{(0,0,j)} k_{(0,i,k)} + a_{(0,0,k)} k_{(0,i,j)} &&\;\; \text{ for all } k>j>i>0, \\
		3 k_{(i,i,j)} &= a_{(0,0,j)} k_{(0,i,i)} + a_{(0,0,i)} k_{(0,i,j)} && \;\;\text{ for all } i,j>0, i\neq j, \\
		3 k_{(i,i,i)} &= a_{(0,0,i)} k_{(0,i,i)} &&\;\; \text{ for all } i>0.
		\end{alignat*}
		The claim can be deduced directly from this.
	\end{proof}

	\subsection{Second Blow-up}\label{subsec:twoblow} Let $V_2 := \Bl_{B_1} V_1$. This is smooth because so is $B_1$. We denote $\pi_2:V_2\to V_1$ the blow-up map, and respectively $\Tilde{E}_1,S_2,P_2,L_2$ the proper transforms of $E_1,S_1,P_1,L_1$. Moreover, we define $B_2 \coloneqq \Tilde{E}_1 \cap E_2 = \PP(N_{B_1}E_1)$, where $E_2$ denotes the exceptional divisor in $V_2$.
	
	\begin{crd2}\label{rmk:coord_V2}
	Let $(a_{(0,0,1)}, \ldots, a_{(0,0,n)}, a', b_J ,[c_a,c_H])$ denote coordinates for the product space $(D(a_{(0,0,0)}) \cap D(b_{(0,1,1)})) \times \PP^{r-1}$, where $r$ is the codimension of $B_1$ as subvariety of $V_0$ and $J$ denotes all multi-indices $(i,j,k)\in[n]^3$ with $i \leq j \leq k$ such that $i$ and $j$ are not both equal to $0$. Thanks to Lemma~\ref{lemma:equations_B1},  the blow-up $V_2$ in the open chart $D(a_{(0,0,0)}) \cap D(b_{(0,1,1)})$ is a closed subvariety given by the equations
    \begin{equation*}
        c_a f_H' - a' c_H = 0, \qquad
        c_{H_1} f_{H_2}' - c_{H_2} f_{H_1}' = 0,
    \end{equation*}
    for $H, H_1, H_2$ running over all $(i,j,k)$ with $k \geq j \geq i \geq 1$. We can choose the affine open of $V_2$ given by $D(c_{(1,1,1)})$, then these equations simplify to
    \begin{equation*}
        c_a f_{(1,1,1)}' - a' =0, \qquad c_H f_{(1,1,1)}' - f_H'=0,
    \end{equation*}
    where $H$ varies as above but we exclude $H = (1,1,1)$. Introducing the new variable $b' \coloneqq f_{(1,1,1)}'$, essentially carrying the same information as $b_{(1,1,1)}$, this affine open of $V_2$ has affine coordinates $(a_{(0,0,i)}, b_{(0,j,k)},b',c_a, c_H)$ with $H \neq (1,1,1)$ subject to no relations. In these coordinates, the equation for $E_2$ in $V_2$ becomes $b' = 0$ and the equation for the proper transform $\Tilde{E_1}$ becomes $c_a = 0$.
    We will always exclude the index $H=(1,1,1)$ when considering the affine coordinates.
    \end{crd2}

	\begin{lemma}\label{lem:iden}
    Write $N_2 \coloneqq N_{\PP(W)} \PP(\Sym^2(W))$ and $N_3 \coloneqq N_{\PP(W)} \PP(\Sym^3(W))$ and let $p_1: B_1 \rightarrow B_0$ be the restriction of the canonical map from the projective bundle $E_1 = \PP(N_{B_0} V_0)$ to its base $B_0 \cong \PP(W)$. Then there is a natural isomorphism
    \begin{align*}
        N_{B_1}E_1 &= p_1^\ast(N_3/N_2) \otimes_{\mathcal{O}_{B_1}}  \mathcal{O}_{B_1}(1) \\ &= p_1^\ast \left(\frac{\Sym^3(W) \otimes \mathcal{O}_{\PP(W)}(3)}{\Sym^2(W) \otimes \mathcal{O}_{\PP(W)}(2)} \right) \otimes  \mathcal{O}_{B_1}(1) \\ &= p_1^\ast(\Sym^3(T\PP(W))) \otimes \mathcal{O}_{B_1}(1).
    \end{align*}
    Hence, over a point $(\lambda, q) \in B_1$, the normal space $N_{B_1} {E_1}_{|_{(\lambda,q)}}$ is naturally identified with $\Sym^3(W)/(\lambda \cdot \Sym^2(W))$. Points in $B_2$ can be thought of as triples consisting of a hyperplane $\lambda$ together with a quadric $q$ and a cubic $c$ inside $\lambda$.
    \end{lemma}
    
    \begin{proof}
    The first isomorphism is given by \cite[Proposition 9.13]{eisenbud_harris_2016}.
    The Euler sequences for $T\PP(W)$, $T\PP(\Sym^2(W))$, $T\PP(\Sym^3(W))$ then give the second and third equality.
    \end{proof}
    
    \begin{lemma} The set-theoretical intersection of all proper transforms of the line conditions in $V_2$ is contained in the union of $S_2$ and the smooth variety $B_2 = \Tilde{E}_1 \cap E_2$.
    \end{lemma}

    \begin{proof}
    The variety $S_2$ is clearly a component of the intersection. By Lemma~\ref{lem:tangency}, the line conditions in $V_1$ are generically tangent to $E_1$, and therefore the tangent space of each line condition is contained in the tangent space of $E_1$. Hence, the intersection of the proper transforms of the line conditions with the exceptional divisor $E_2$ is contained in $\Tilde{E_1}$.
    \end{proof}
    
    A similar reasoning as in Lemma~\ref{lemma:liftofphi0} shows also the following.
     \begin{lemma}\label{lemma:phi2}
The lift $\phi_2: \Bl_{\Delta}\PP^n\times\PP^n \to V_2$ of $\phi_1$ is explicitly given by
\begin{equation*}
    \begin{split}
    a_{(0,0,1)}&=3\mu_1 + u_1,\\
    a_{(0,0,i)}&=3\mu_i + s_i u_1 \quad \hbox{ for }i>1,\\
    b_{(0,i,i)}&=s_i^2 \qquad \qquad \;\,  \hbox{ for }i>1,\\
    b_{(0,1,i)}&=2s_i \qquad \qquad \, \hbox{ for }i>1,\\
    b_{(0,i,j)}&=2s_i s_j \qquad \quad \;\hbox{ for }j>i>1,\\
    b'&=-2u_1,
    \end{split}\qquad \qquad \qquad
    \begin{split}
    c_a&=u_1 /\,2, \\
    c_{(1,1,i)} &= s_i \qquad\;\, \hbox{ for }i>1,\\
    c_{(1,i,i)} &= s_i^2 \qquad \,\, \hbox{ for }i>1,\\
    c_{(i,i,j)} &= s_i^2 s_j  \;\;\quad\hbox{ for }i, j> 1,\\
    c_{(i,i,i)} &= s_i^3\qquad \;\, \hbox{ for }i\neq 0,1,\\
    c_{(i,j,k)} &= 2 s_i s_j s_k \;\hbox{ for }k>j>i>0.
    \end{split}
\end{equation*}
\end{lemma}

    Lemma \ref{lem:tangency} implies that the set-theoretic intersection of $S_1$ with $B_1$ is given by $\phi_1(e)$. It is not hard to see then that $S_2$ is isomorphic to $S_1$, hence to $\Bl_{\Delta}\PP^n\times\PP^n$.
    Abusing notation, we will indicate with $e$ the exceptional divisor of $\Bl_{\Delta}\PP^n\times\PP^n$ as well as all its isomorphic images under the maps $\phi_i$.
    
    \begin{lemma}\label{lemma:S2intB2} 
    The following hold:
    \begin{enumerate} 
        \item  $B_2$ intersects $S_2$ along $e$.
        \item  The line conditions in $V_2$ are generically smooth along $B_2$.
    \end{enumerate}
    \end{lemma}
    
    \begin{proof}
    First, recall that $S_1$ is tangent to $E_1$ along $e$. In fact, for any point $p\in e$ we have $T_{\phi_1(p)} S_1 = d\phi_1(T_p (\Bl_{\Delta} \PP^n \times \PP^n))$. Working in the chosen affine chart for $V_1$, since the entry relative to $a'$ in the column vectors of the Jacobian is always zero, then $T_{\phi_1(p)} S_1$ is contained in the tangent space of $E_1$. By invariance under projective transformations this is true everywhere.
    Thus, since $B_1$ intersects $S_1$ along $\phi_1(e)$ we have that $S_2 \cap E_2 \subseteq \Tilde{E_1} \cap E_2 = B_2$ because the tangent space of $S_1$ is contained in the tangent space of $E_1$.
    
    For the second claim, observe that the line conditions in $V_1$ are generically smooth along $B_1$. The claim then follows from the blow-up closure lemma and the fact that the blow-up of a smooth variety is again smooth. 
    \end{proof}

    \begin{rmk}\label{eq:handq}
    A cubic $k \in B_2|_{[\lambda, q]} \cong \PP(\Sym^3(W)/(\lambda \cdot \Sym^2(W)))$, whose defining equation can be uniquely written (up to scaling) in the form $k = k_{(1,1,1)} x_1^3 + k_{(1,1,2)} x_1^2 x_2 + \ldots + k_{(n,n,n)} x_n^3$, not containing any monomial divisible by $x_0$, is identified with the projective coordinates $[c_a,c_H]$ in Remark~\ref{rmk:coord_V2} via $c_a = 0$ and $k_H = 3 c_H$ for all indices $H = (i,j,k)$ with $k \geq j \geq i \geq 1$ and $|H| > 1$ and $k_{(i,i,i)} = c_{(i,i,i)}$ for all $i \geq 1$. In particular, $S_2 \cap B_2$ consists of all triples $(\lambda, q, k) = (\lambda, g^2, g^3)$ for some hyperplane $g$ in $\PP(W/\lambda)$ as follows from the explicit description of $\phi_2$ in Lemma~\ref{lemma:phi2}.
    \end{rmk}

\begin{prop}\label{prop:inters} 
Let $\overline{\lambda}\coloneqq([\lambda],[q],[k])$ be a point of $B_2$, i.e. a hyperplane $\lambda$ together with a quadric $q$ and a cubic $k$. Consider the line condition $L_2^{\ell}$ in $V_2$ corresponding to a line $\ell\subseteq\PP(W)$. Then:
\begin{enumerate}
    \item $\ell$ intersects $\lambda$ at the quadric $q$ if and only if $L_2^{\ell}$ is tangent to $E_2$ at $\overline{\lambda}$;
    \item $\ell$ intersects $\lambda$ at the cubic $k$ if and only if $L_2^{\ell}$ is tangent to $\Tilde{E_1}$ at $\overline{\lambda}$;
\end{enumerate}

\begin{proof}
We can assume the hyperplane $\lambda$ to be $\mathcal{V}(x_0)$ and $\ell$ the line $\mathcal{V}(x_1,x_3,\dots, x_n)$. By plugging in the equations $c_{(1,2,2)} b'-3 b_{(2,2,2)}+2 a_{(0,0,2)} b_{(0,2,2)}=0$ and $a'-c_a b'=0$ in the equation of the proper transform of the line condition in $V_1$, we get the equation for $L^{\ell}$ in local coordinates in $V_3$, i.e. 
\begin{equation*}
    4b_{(0,2,2)}^3 c_a + c_{(2,2,2)}^2 b'
\end{equation*}
From Lemma~\ref{lemma:equations_B1}, one has that the quadrics intersecting $\lambda$ at its point of intersection with $\ell$ are given by the equation: $b_{(0,2,2)} = 0$.
From Remark~\ref{eq:handq} the cubics intersecting $\lambda$ in $\lambda \cap \ell$ are given by the equation $c_{(2,2,2)} = 0$.
The statement on the tangency at $E_2$ and at $\Tilde{E}_1$ follows from the direct computation with the equations.
\end{proof}
\end{prop}

\begin{rmk}\label{rem:notcubsmooth}
We can notice that if the line $\ell$ does not intersect the quadric $q$ or the cubic $k$ at the point $\overline{\lambda}$, than the line condition $L^{\ell}$ is smooth at  $\overline{\lambda}$. This is clear from the proof of the previous lemma when $\lambda = x_0$ and $\ell=\mathcal{V}(x_1,x_3,\dots, x_n)$. The claim follows by invariance under projective transformations.
\end{rmk}

	\subsection{Third Blow-up}\label{subsec:thirdblow}
	Let $V_3 \coloneqq \Bl_{B_2} V_2$. This is smooth because $B_2$ is. We stick to the notation $\pi_3:V_3\to V_2$ for the blow-up map and $E_3$ for the exceptional divisor. We denote $L_3$ the proper transform in $V_3$ of the a line condition $L_2\subseteq V_2$, and $S_3$ is the proper transform of $S_2$. 
	
	\begin{crd3}\label{eq:}
	In the chosen chart for $V_2$ described in Remark~\ref{rmk:coord_V2} the base locus $B_2$ is given by $\mathcal{V}(c_{a},b')$.
	Consider \sloppy{$(D(a_{(0,0,0)})\cap D(b_{(0,1,1)})\cap D(c_{(1,1,1)}))\times\PP^1$} with coordinates $(a_{(0,0,i)},b_{(0,j,k)},b',c_{a},c_H,[d_{c},d_{b}])$. The blow-up of $B_2$ in the chosen chart of $V_2$ can be described as the subvariety determined by
	\[b' d_c = d_b c_a.\]
	In the affine chart $D(a_{(0,0,0)})\cap D(b_{(0,1,1)})\cap  D(c_{(1,1,1)})\cap D(d_c)$ of $V_3$ we can work with coordinates $(a_{(0,0,i)},b_{(0,j,k)},c_{a},c_H,d_{b})$. The exceptional divisor $E_3$ is cut out by $c_{a}=0$ in this chart.
	\end{crd3}
	
	\begin{rmk}\label{rmk:invlinecond} The line condition $L_3^{\ell}$ corresponding to $\ell:=\mathcal{V}(x_1,x_3,\dots, x_n)$ has equation
    \begin{equation*}
        4b_{(0,2,2)}^3 + c_{(2,2,2)}^2 d_b=0
    \end{equation*}
    and therefore every other line condition obtained by this one by an induced action of the $\PGL_n$-action preserving this chart will be of the type
    \begin{equation*}
        4f(b_{J})^3 + g(c_{H})^2 d_b=0
    \end{equation*}
    where $f$ is a linear function in the $b_{J}$ coordinates and $g$ is a linear function in the $c_{H}$ coordinates.
	\end{rmk}
    
    We now prove that the intersection of all line conditions coincides with $S_3$.
    
    \begin{prop} \label{prop:intersectionV3}
    The intersection of all line conditions in $V_3$ is supported on the smooth irreducible variety $S_3$.
    \end{prop}
    \begin{proof}
        The base locus $B_2=\PP(N_{B_1}E_1)$ has codimension $2$ in $V_2$. The exceptional divisor $E_3=\PP(N_{B_2}V_2)$ is then a $\PP^1$-bundle over $B_2$.
        Let $\overline{\lambda}\coloneqq([\lambda],[q],[k])$ be a fixed point in $B_2$ with $\pi_2\circ\pi_1(\overline{\lambda})=[\lambda^3] \in B_0$, i.e., a hyperplane $\lambda$ together with a quadric $q$ and a cubic $k$ lying on $\lambda$.
        Thanks to Remark~\ref{rem:notcubsmooth}, a general line condition is smooth at $\overline{\lambda} \in B_2$, has codimension one, and contains $B_2$. Its proper transform intersects the fiber of $\PP(N_{B_2}V_2)$ over $\overline{\lambda}$ at most in one point. 
        We need to check that line conditions in $V_3$ can only intersect in $E_3$ above $B_2\cap S_2$.
       
       The base locus $B_2=E_2\cap \tilde{E_1}$ is smooth of codimension $2$ in $V_2$. Therefore, the proper transforms of $\Tilde{E_1}$ and $E_2$ in $V_3$ cut the fiber of $E_3$ over any $\overline{\lambda}\in B_2$ in different points $r_1$ and $r_2$.
       From Proposition \ref{prop:inters} follows that if a line $\ell$ intersects $q$, then the line condition $L^{\ell}_3$ contains $r_2$, while if $\ell$ intersects $k$, then the line condition $L^{\ell}_3$ contains the point $r_1$. 
       
       We claim that in order for the line conditions to intersect over $\overline{\lambda}$ we must have $q=hg$ and $k=h^2g$ where $h,g$ are linear forms on the hyperplane $\lambda$. In fact, suppose there is a point of $q$ which is not in $k$. Then, we can take a line $\ell$ in $\PP^n$ passing through that point and not contained in $\lambda$. Thanks to Remark~\ref{rem:notcubsmooth}, the line condition $L^{\ell}_2$ is smooth at $\overline{\lambda}$ and $L^{\ell}_3$ intersects the fiber over $\overline{\lambda}$ in a unique point, necessarily in $r_2$. Take now another line condition $L^{\ell'}_2$ in $V_2$ such that the line $\ell'$ does not intersect the cubic nor the quadric. The line condition $L^{\ell'}_2$ is a hypersurface which is smooth at $\overline{\lambda}$ and contains $B_2$. If its proper transform intersects the fiber over $\overline{\lambda}$ in $r_2$, then $L^{\ell'}_2$ is tangent to $E_2$, and by Proposition \ref{prop:inters} it must intersect the quadric.
   
       Similarly, we can show that there is no point of $q$ which is not in $k$. Hence we proved that in order for the line conditions to intersect over $\overline{\lambda}$ we must have $q=k$ set-theoretically. But this is equivalent to $q=hg$ and $k=h^2g$ with $h,g$ linear forms on the hyperplane $\lambda$.
       
       By Remark~\ref{eq:handq}, we just have to show that $g=h$. It is enough to show it for $\lambda=x_0$ because the locus $B_2\cap S_2$ is invariant under the induced $\PGL_n$-action on $V_2$.
       Consider the point  $\overline{x_0}=([x_0^3],[q],[k])$, where 
       \[q=(h_1x_1+\dots+h_nx_n)(g_1x_1+\dots + g_n x_n), \quad k=(h_1x_1+\dots+h_nx_n)^2(g_1x_1+\dots + g_nx_n)\]
       are respectively a quadric and a cubic on the hyperplane $x_0=0$. 
       
       We claim there exists an index $l$ such that $\overline{x_0}$ belongs to $D(b_{(0,l,l)})$. First, fix $i$ and $j$ such that $\overline{x_0}$ belongs to the affine chart $D(b_{(0,i,j)})\cap D(c_{(i,i,j)})$. Then, we can work in this affine chart with its coordinates. For $t \in \mathbbl{k}$, consider the line conditions $L^{x_i+tx_j}_2$ in $V_2$ corresponding to the line given by the vanishing of $x_i+tx_j=0$ and of all coordinates except for $x_0,x_i,x_j$. Their equations in $V_2$ are
       \[4(t^2b_{(0,i,i)}+b_{(0,j,j)}-t)^3c_{a}+(t^3c_{(i,i,i)}-3t^2-c_{(j,j,j)}+3tc_{(i,j,j)})^2b'=0\]
       in the mentioned open chart, where $b'=f_{(i,i,j)}'$ and $a'=f_{(0,i,j)}$.
       Notice that thanks to Lemma \ref{lemma:equations_B1} and Lemma \ref{eq:handq}, we have a relation between the coordinates and the coefficients of $q$ and $k$.
       Suppose $b_{(0,l,l)}(\overline{x_0})=h_lg_l=0$ for all indices. 
       The line conditions $L^{x_i+tx_j}_3$ intersect the fiber over $\overline{x_0}$ in $V_3$ in the points 
       \[[-4t^3,(3t^2-3tc_{(i,j,j)}(\overline{x_0}))^2].\]
       Recall that we are working in a chart such that $h_i^2g_j=c_{(i,i,j)}(\overline{x_0})\neq 0$, therefore $c_{(i,j,j)}(\overline{x_0})=h_j^2g_i=0$ and the points become
       \[[-4t^3,9t^4]=[-4,9t].\] 
       This is absurd because we are assuming these line conditions to intersect over $\overline{x_0}$ and we proved the claim. Without loss of generality, we put $l=1$, and we work in the affine chart  $D(a_{(0,0,0)})\cap D(b_{(0,1,1)})\cap D(c_{(1,1,1)})$ of $V_2$, so we can assume $h_1=g_1=1$.
       
       We claim that $h_i$ is zero if and only if $g_i$ is zero for every index $i$. Suppose there exists an index $i$ such that $h_i=0$ and $g_i\neq 0$ and consider the line conditions $L^{x_i+tx_1}_2$ in $V_2$ with equations
       \[4(t^2b_{(0,i,i)}+1-tb_{(0,1,i)})^3 c_{a}+(t^3c_{(i,i,i)}-3t^2c_{(i,i,1)}-1+3tc_{(1,1,i)})^2 b'=0\]
       in the chosen affine chart. The line conditions $L^{x_i+tx_1}_3$ intersect the fiber over $\overline{x_0}$ in $E_3$ in the points
       \begin{align*}
       [4(1-t b_{(0,1,i)}(\overline{x_0}))^3,(1-3t c_{(1,1,i)}(\overline{x_0}))^2]&=[4(1-t h_i-tg_i)^3,(1-tg_i)^2]\\
       &=[4(1-t g_i)^3,(1-t g_i)^2].
       \end{align*}
       Notice that for $g_i\neq 0$, these would give different points for different values of $t$ which is absurd. The same reasoning holds for $g_i=0$ and $h_i\neq 0$ proving then the claim.
       Finally, if we consider the line condition $L^{x_j}_2$ in $V_2$ corresponding to the line given by the vanishing of all coordinates except for $x_0$ and $x_j$ then this has equation
       \[4b_{(0,j,j)}^3c_a+c_{(j,j,j)}^2b'=0\]
       in the chosen open chart. If we assume  $(b_{(0,j,j)}(\overline{x_0}),c_{(j,j,j)}(\overline{x_0}))\neq(0,0)$, we must have
       \begin{equation*}
           \begin{split}
           [4,1]=[4(b_{(0,j,j)}(\overline{x_0}))^3,(c_{(j,j,j)}(\overline{x_0}))^2] \Leftrightarrow &
           (b_{(0,j,j)}(\overline{x_0}))^3=(c_{(j,j,j)}(\overline{x_0}))^2 \\
            \Leftrightarrow &
           g_j^3h_j^3=g_j^4h_j^2 
           \end{split}
       \end{equation*}
       and therefore $g_j=h_j$ when $h_j$ is non-zero.
    \end{proof}
    
    Since $S_3$ will be the next center for the blow-up, we denote it with $B_3$. From Lemma~\ref{lemma:S2intB2} follows that $B_3$ is isomorphic to $S_2$. In particular, the isomorphism map $\phi_2:\Bl_{\Delta}\PP^n\times\PP^n\to S_2$ defined in Lemma~\ref{lemma:phi2} lifts to the following map.
	
	\begin{lemma}\label{lemma:phi3}
    The lift $\phi_3: \Bl_{\Delta}\PP^n\times\PP^n \to V_3$ of $\phi_2$ on the chosen open charts is explicitly given by
    \begin{equation*}
        \begin{split}
        a_{(0,0,1)}&=3\mu_1 + u_1,\\
        a_{(0,0,i)}&=3\mu_i + s_i u_1 \quad\hbox{ for }i>1,\\
        b_{(0,i,i)}&=s_i^2  \;\qquad\qquad\, \hbox{ for }i>1,\\
        b_{(0,1,i)}&=2s_i \qquad\qquad\hbox{ for }i>1,\\
        b_{(0,i,j)}&=2s_i s_j \qquad\quad\,\hbox{ for }j>i>1,\\
        c_a&={u_1}/ \, {2},
        \end{split}\qquad \qquad \qquad
        \begin{split}
        c_{(1,1,i)} &= s_i \qquad\;\;\; \hbox{ for }i>1,\\
        c_{(1,i,i)} &= s_i^2  \quad\quad\;\;\,\,\hbox{ for }i>1,\\
        c_{(i,i,j)} &= s_i^2 s_j \quad\;\;\;\, \hbox{ for }i, j> 1,\\
        c_{(i,j,k)} &= 2 s_i s_j s_k \;\;\hbox{ for }k>j>i>0, \\
        d_b &=-4. 
        \end{split}
    \end{equation*}
    \end{lemma}
    
    \begin{rmk}\label{rmk:equB3}
    The equations
    \begin{alignat*}{2}
     & d_b+4 = 0, \\
    g_{(0,1,i)}&:=b_{(0,1,i)}-2c_{(1,1,i)} = 0  \quad\qquad &&\hbox{ for }i>1,\\
    g_{(0,i,i)}&:=b_{(0,i,i)}-c_{(1,1,i)}^2 = 0 \quad\qquad  &&\hbox{ for }i>1,\\
    g_{(0,i,j)}&:=b_{(0,i,j)}-2c_{(1,1,i)}c_{(1,1,j)} = 0 \quad\qquad  &&\hbox{ for }j>i>1,\\
    g_{(1,i,i)}&:=c_{(1,i,i)}-c_{(1,1,i)}^2 = 0 \quad\qquad &&\hbox{ for }i> 1,\\
    g_{(i,i,j)}&:=c_{(i,i,j)}-c_{(1,1,i)}^2c_{(1,1,j)} = 0 \quad\qquad &&\hbox{ for }i,j>1,\\
    g_{(i,j,k)}&:=c_{(i,j,k)}- 2 c_{(1,1,i)}c_{(1,1,j)}c_{(1,1,k)} = 0\quad\qquad  &&\hbox{ for }k>j>i>0.
    \end{alignat*}
    cut out $B_3$ in the chosen affine open chart. 
    \end{rmk}

	\subsection{Fourth Blow-up}\label{subsec:fourblow}
	Recall that $B_3=S_3$. Let $V_4\coloneqq\Bl_{B_3}V_3$.
    We will write $E_4$ for the exceptional divisor and $\pi_4: V_4 \to V_3$ for the blow-up map.
    
    \begin{crd4}\label{coordinates:4}
	In the chosen affine chart of $V_3$ the base locus $B_3$ is cut out by the equations in Remark~\ref{rmk:equB3}.
    Consider $D(a_{(0,0,0)})\cap D(b_{(0,1,1)})\cap D(c_{(1,1,1)})\cap D(d_c)\times \PP^{\binom{n+3}{3}-2n-2}$ with coordinates $(a_{(0,0,i)},b_{(0,j,k)},c_{a},c_H,d_{b},[e_d,e_F])$ where $F$ is the set of indices $(i,j,k)$ with $k\geq j\geq i\geq 0$ and $j>1$. The blow-up of $V_3$ along $B_3$ in the chosen affine chart can be described as the subvariety determined by
    \begin{alignat*}{2}
        &e_d g_{(i,j,k)} - (d_b+4) e_{(i,j,k)} = 0 \;\;\quad&& \hbox{ for } (i,j,k)\in F,\\
        &e_{(i_1,j_1,k_1)} g_{(i_2,j_2,k_2)} - g_{(i_1,j_1,k_1)} e_{(i_2,j_2,k_2)} = 0 \;\;\quad&& \hbox{ for } (i_1,j_1,k_1),(i_2,j_2,k_2)\in F.
    \end{alignat*}
	In the affine chart $D(a_{(0,0,0)})\cap D(y_{(0,1,1)})\cap D(c_{(1,1,1)})\cap D(d_c)\cap D(e_{(0,1,2)})$ of $V_4$ we can work with coordinates $(a_{(0,0,i)},c_{a},c_{(1,1,2)},..,c_{(1,1,n)},e_d,e_F,e')$ where $e'=g_{(0,1,2)}$ is used as a coordinate and $F$ is the same index set as above but we exclude $(0,1,2)$. The exceptional divisor $E_4$ is cut out by $e'=0$ in this chart.
	\end{crd4}
	
	\begin{prop}\label{prop:vector_bundle}
    The intersection of all line conditions in $V_4$ is supported on a smooth subvariety $B_4$ of codimension $\binom{n+2}{3}$ inside $E_4$. More precisely, $B_4=\PP(\mathcal{E})$  where $\mathcal{E}$ is a subvector bundle of rank $\binom{n}{2}$ of the normal bundle $N_{B_3} V_3$ .
    \end{prop}

    \begin{proof}
    We generalize the proof of \cite[Proposition 4.1]{Aluffi:1990}. Let $R_{\mu} \subseteq V_0$ denote the subvariety of cubics containing the hyperplane $\mu$. Clearly, $R_\mu \cong \PP(\Sym^2(W))$ is smooth. By Lemma \ref{lem:lemma0.1}, a line condition $L^{\ell}$ is smooth at $[\lambda \mu^2] \in S_0 \setminus B_0$ if the line $\ell$ intersects $\mu$ in a single point outside $\lambda$. Clearly, $T_{[\lambda \mu^2]} R_\mu \subseteq T_{[\lambda \mu^2]} L^{\ell}$ for every line $\ell$, and Lemma \ref{lem:lemma0.1} shows that
    \begin{equation*}\label{eq:intersection_tangent_spaces_line_conditions}
    \bigcap_{\ell \subseteq \PP(W) \text{ line}} T_{[\lambda \mu^2]} L^{\ell} = T_{[\lambda \mu^2]} R_\mu.
    \end{equation*}
    Clearly, finitely many lines suffice for the intersection of the tangent spaces to agree with $T_{[\lambda \mu^2]} R_\mu$ over every point $[\lambda \mu^2] \in B_3 \setminus e \cong S_0 \setminus B_0$. By Proposition~\ref{prop:intersectionV3}, the intersection of the proper transforms $L_{3}^{\ell}$ in $V_3$ for all lines $\ell$ agrees set-theoretically with $S_3 = B_3$. The proper transforms $L^{\ell}_{4}$ in $V_4$ therefore only intersect in the exceptional divisor $E_4$. We claim that their intersection is precisely the projectivization of a vector subbundle $\mathcal{E} \subseteq N_{B_3} V_3$. We construct $\mathcal{E}$ as the intersection of the images of the tangent sheaves $\mathcal{T} L^{\ell}_{3}|_{B_3}$ in $N_{B_3} V_3$ corresponding to finitely many lines $\ell$. The finiteness will ensure that the resulting subsheaf $\mathcal{E}$ of $N_{B_3} V_3$ is coherent. First, we pick finitely many lines such that the intersection of the tangent spaces over every point $[\lambda \mu^2] \in B_3 \setminus e$ agrees with $T_{[\lambda \mu^2]} R_\mu$. The intersection of the images of the tangent sheaves in $N_{B_3} V_3$ of these line conditions defines a coherent subsheaf $\mathcal{E}'$ which restricts to a subvector bundle over $B_3 \setminus e$. Then by Lemma~\ref{lem:lemma0.1} and a Zariski closure argument, every other line condition $L^{\ell}_{4}$ contains the projectivization $\PP(\mathcal{E}'|_{B_3 \setminus e})$, and we have 
    \[\mathcal{E}'([\lambda \mu^2]) \cong T_{[\lambda\mu^2]} R_{\mu} / T_{[\lambda\mu^2]} S_0,\] 
    where $\mathcal{E}'(p)$ denotes the geometric fiber of $\mathcal{E}'$ over the point $p$. The rank of $\mathcal{E}'$ over $B_3 \setminus e$ is $r = \binom{n+2}{2} - 2n - 1 = \binom{n}{2}$. Next, we fix a point $p \in e = B_3 \cap E_3$ lying in our affine open chart. By Remark~\ref{rmk:invlinecond}, in the chosen affine chart the equation for $L^{{\ell}}_{3}$ with ${\ell}$ any line passing through the point $[1:0:\cdots:0]$ does not depend on the variable $c_{a}$, and the equation determining $E_3$ in $V_3$ is exactly $c_a = 0$. The transversality of such line conditions can therefore be checked outside of $E_3$ and hence in $S_0 \setminus B_0$. This shows at once that there are $\codim(R_\mu, V_0) = \binom{n+2}{3}$ lines ${\ell}_i \subseteq \PP(W)$ such that the line conditions $L^{{\ell}_i}_{3}$ are all smooth and intersect transversally at $p$. Moreover, employing the $\PGL_n$-action and using that it acts transitively on $e$ by Lemma~\ref{lemma:PGL_transitive_on_e}, we obtain finitely many more lines such that the intersection of their tangent spaces at every point of $e$ has dimension at most $r$. Let $\mathcal{E}$ be the intersection of $\mathcal{E}'$ with the images of the tangent sheaves in $N_{B_3} V_3$ of these new line conditions. Then $\mathcal{E}$ is a coherent subsheaf of $N_{B_3} V_3$ which still restricts over $B_3 \setminus e$ to a vector subbundle of rank $r$ and has rank $\leq r$ over every point of $e$. By upper semi-continuity of the rank, since $\mathcal{E}$ is coherent, $\mathcal{E}$ is a subvector bundle of $N_{B_3} V_3$ of rank $r$ everywhere. As $\PP(\mathcal{E})$ is an irreducible closed subset of $V_4$, a Zariski closure argument then shows that it is contained in $L^{\ell}_4$ for every $\ell$, so it is contained in the intersection of all line conditions in $V_4$. Nevertheless, by construction $\PP(\mathcal{E})$ also contains the intersection of some (and hence of all) line conditions in $V_4$, proving equality.
    \end{proof}

	\subsection{Fifth Blow-up}
	Let $V_5 \coloneqq \Bl_{B_4}V_4$. 
    Denote with $E_5$ the exceptional divisor and let $\pi_5: V_5 \to V_4$ be the blow-up map.
    Let $\tilde{E_4}$ be the strict transform of $E_4$.
    
    \begin{lemma}\label{lemma:NB4V4iso}
    We have the isomorphism
    \begin{equation*}
        N_{B_4} E_4 \cong (\pi_4|_{B_4})^\ast(N_{B_3} V_3/\mathcal{E}) \otimes \mathcal{O}_{B_4}(1).
    \end{equation*}
    Moreover, over $U:=B_4 \setminus (\pi_4|_{B_4})^{-1}(e)$ the normal bundle $N_{B_4} E_4$ restricts to
    \begin{equation*}
        N_{U} E_4 \cong (\pi_4|_{U})^\ast\left(\frac{\Sym^3(W) \otimes \mathcal{O}(1,2)}{\Sym^2(W) \otimes \mathcal{O}(1,1)} \right) \otimes \mathcal{O}_{U}(1),
    \end{equation*}
    where $\mathcal{O}(a,b)$ denotes the pullback to $\PP^n \times \PP^n \setminus \Delta$. In particular, the fiber of $N_{B_4} E_4$ over some point of $B_4 \setminus \pi_4^{-1}(e)$ mapping to $[\lambda \mu^2] \in B_3 \setminus e$ is given by $\Sym^3(W)/(\mu \cdot \Sym^2(W))$.
    \end{lemma}
    
    The proof is similar to that of Lemma~\ref{lem:iden}. We can now start to understand the intersection of all line conditions inside $V_4$.
    
    \begin{lemma} Fix a line $\ell$ of $\PP^n$ and a cubic $\lambda\mu^2$ such that $\ell$ does not intersect $\lambda\cap \mu$. The strict transform $L_5^{\ell}$ in $V_5$ contains a point $p$ in $E_5\cap \tilde{E}_4$ with $(\pi_4 \circ \pi_5)(p)= [\lambda\mu^2]$ if and only if the line $\ell$ intersects the cubic on $\mu$ associated with $p$, i.e. the element of $\Sym^3(W)/(\mu \cdot \Sym^2(W))$.
    \end{lemma}
    
    \begin{proof}
    By assumption, $L_3^{\ell}$ and its proper transforms are smooth at every point over $[\lambda \mu^2] \in B_3$. We have $(L_5^{\ell} \cap \tilde{E}_4 \cap E_5)|_{\pi_5(p)} = \PP(N_{B_4} (L_4 \cap E_4)|_{\pi_5(p)})$. Since $L_4 \cap E_4|_{U} = \PP(N_{B_3} L_3|_{U})$ on the smooth locus $U$ of $L_3^{\ell}$ inside $B_3$, we have the canonical isomorphisms
    \begin{equation*}
        N_{B_4}(L_4^{\ell}\cap E_4)|_{\pi_5(p)} \cong ((\pi_4|_{B_4})^\ast(N_{B_3} L_3^{\ell}/\mathcal{E}) \otimes \mathcal{O}_{B_4}(1))|_{\pi_5(p)} \cong (N_{B_3} L_3^{\ell}/\mathcal{E})|_{[\lambda \mu^2]}.
    \end{equation*}
    Knowing that $T_{[\lambda \mu^2]}L_3^{\ell}$ is given by those cubics containing $\ell \cap \mu$ by Lemma~\ref{lem:lemma0.1}(iii) and that the fiber of $\mathcal{E}$ at $[\lambda \mu^2]$ is the quotient of the cubics containing $\mu$ by the tangent space of $B_3$ at $[\lambda\mu^2]$,  we conclude that the projective fiber of this bundle over the point  $\pi_5(p)$ of $B_4$ is exactly given by those cubics on $\mu$ touching $\ell$.
    \end{proof}
    
    \begin{lemma}\label{lem:lineCondE4}
    There exists a point $[\lambda\mu^2]$ with $\lambda\neq\mu$ in $B_3$ such that for every point $\overline{\lambda\mu^2}$ in $B_4$ with $\pi_3(\overline{\lambda\mu^2}) = [\lambda\mu^2]$ the intersection of the line conditions in the fiber $(E_5)|_{\overline{\lambda\mu^2}}$ is contained in the proper transform $\tilde{E}_4$ of $E_4$ in $V_5$.
    \end{lemma}
    
    \begin{proof}
    Consider the chart in $V_3$ given by $D(a_{(0,0,0)}) \cap D(b_{(0,1,1)}) \cap D(c_{(1,1,1)})\cap D(d_a)$.  We can now choose any point in $B_3\setminus e$; we will choose our favourite one $P:=[(x_2+x_0)^2(x_2+x_1+x_0)]$. Notice that this is indeed contained in the chart. 
    We will denote with $(P,Q)$ a point in the fiber  of $B_4$ over $P$, where $Q\in\PP (R_{(x_2+x_0)}/(TB_3)_{P})$ and where $R_{(x_2+x_0)}$ is the space of cubics which are divisible by $(x_2+x_0)$. Points $P_{\epsilon,q}$ in $R_{(x_2+x_0)}$ can be uniquely written up to constants as
    $P_{\epsilon,q}:=(x_2+x_0)q$
    in the projective coordinates $[q_{ij}]_{i,j\in[n]}$ of the quadric $q$ in $(n+1)$ variables. In this coordinates, the tangent space $(TB_3)_{P}$ is given by
    \begin{equation*}
    \begin{cases}
    q_{00}+q_{22}=q_{02}\\
    q_{0j}-q_{2j}=0\quad \hbox{ for all } j\neq{0,1,2},\\
    q_{ij}=0 \qquad\quad\;\,\hbox{ for all } i,j\neq{0,1,2}. 
    \end{cases}
    \end{equation*}
    Denoting 
    $\pi_P:R_{(x_2+x_0)}\setminus{(TB_3)_P}\to (B_4)_P$ the quotient map followed by the projectivization, every point $(P,Q)$ can be represented in a non-unique way as $\pi_P([q_{ij}]_{i,j\in[n]})$. We now want to show that for every point  $(P,Q)=\pi_P([q_{ij}]_{i,j\in[n]})$ there exists a sequence of line conditions $L_4^m$ in $V_4$ which are smooth at $(P,Q)$ and such that the hyperplanes $(TL_4^{m})_{(P,Q)}$ tends to $(TE_4)_{(P,Q)}$ as subvector spaces of $(TV_{4})_{(P,Q)}$. This proves the lemma, as the intersection of all line conditions in $(E_5)_{(P,Q)}$ will be the same as the intersection of all line condition and $(\tilde{E}_4)_{(P,Q)}$.

    Before choosing appropriate line conditions, let us compute the projective coordinates $[e_{d},e_{(0,1,2)},\dots,e_{(0,n,n)}, e_{(1,1,2)}, \dots,e_{(n,n,n)}]$ for $(P,Q)=\pi_P([q_{ij}]_{i,j\in[n]})$ as functions of $[q_{ij}]$. We get the following coordinates for the point $(P,Q)$:
    \begin{equation*}
        \begin{split}
        {e_d}&= 0, \\
        {e}_{(0,1,j)}&=0  \qquad\qquad\;\,\,\quad\hbox{ for }j\neq 0,1,\\
        {e}_{(0,2,2)} &=3(q_{02}-q_{00}-q_{22}),  \\
        {e}_{(0,j,j)} &=-3q_{jj} \qquad\;\;\,\quad\hbox{ for }j\neq 0,1,2,  \\
        {e}_{(0,2,j)}&=3(q_{0j}-q_{2j})\quad \hbox{ for }j\neq 0,1,2,\\
        {e}_{(0,i,j)}&=-3q_{ij}\qquad\quad\;\;\, \hbox{ for }i,j\neq 0,1,2,\\
        {e}_{(1,2,2)}&=\frac{3}{2}(q_{02}-q_{00}-q_{22}), \\
        \end{split}\qquad\qquad
        \begin{split}
        {e}_{(1,j,j)}&=-\frac{3}{2}q_{jj} \qquad\quad\,\hbox{ for }j\neq 0,1,2,  \\
        {e}_{(1,2,j)}&=\frac{3}{2}(q_{0j}-q_{2j}) \;\;\hbox{ for }j\neq 0,1,2,\\
        {e}_{(1,i,j)}&=-\frac{3}{2}q_{ij}  \hspace{1.1cm}\hbox{ for }i,j\neq 0,1,2,\\
        {e}_{(i,i,j)} &=0 \hspace{1.8cm}\hbox{ for }i,j\neq 0,1,\\
        {e}_{(j,j,j)} &=0 \hspace{1.8cm}\hbox{ for }j\neq 0,1.\\
        \end{split}
    \end{equation*}
    Notice that this makes sense as long as $[q_{ij}]_{i,j\in[n]}\notin
    (TB_3)_P$, which is the case we are interested in.
    
    We will use the notation $L^{j,t}$ for line conditions associated to the lines 
    \begin{equation*}
    \mathcal{V}(x_1+tx_j,x_2,\dots,\hat{x}_j,\dots, x_n).
    \end{equation*}
    The proper transform $L^{j,t}_3$ of these line conditions in the chosen affine chart for $V_3$ are given by 
    \begin{equation*}
       4(t^2+b_{(0,j,j)}-tb_{(0,1,j)})^3+(t^3-3t^2c_{(1,1,j)}-c_{(j,j,j)}+3tc_{(1,j,j)})^2d_b=0.
    \end{equation*}
    Notice that the line condition $L^{j,0}_3$ is singular at $P$ for any $j\neq 0,1$, but the line conditions $L^{j,t}_3$ for $t\neq 0$ are not, and therefore the proper transforms $L^{j,t}_4$  in $V_4$ are smooth at every point $(P,Q)\in B_4$.  
    Now consider the proper transform of such line condition in a chart of $V_4$ different from $D(e_d)$ with coordinate $e'$. Notice that we can do that because $e_d(P,Q)=0$ for every choice of $Q$. Since we are interested in the gradient of the equation evaluated on points in $B_4\subseteq E_4=\{e'=0\}$), we can just look at the gradient of the following equation:
    \par
    {\footnotesize
    \begin{align*}
    \big(12(t-c_{(1,1,j)})^4(e_{(0,j,j)}-t e_{(0,1,j)})+e_d(t-&c_{(1,1,j)})^6- 8(t-c_{(1,1,j)})^3(3te_{(1,j,j)}-e_{(j,j,j)})\big)+\\
    e'\big(12(t-c_{(1,1,j)})^2(e_{(0,j,j)}-te_{(0,1,j)}&)^2+2e_d(t-c_{(1,1,j)})^3(3te_{(1,j,j)}-e_{(j,j,j)})-\\
    4(3te_{(1,j,j)}&-e_{(j,j,j)})^2\big)=0.
    \end{align*}}
    \par
    \noindent If we look at partial derivatives $\partial_{y}$ with respect to variables $y\neq e'$ evaluated at the point $(P,Q)$, we have that $\frac{\partial_{y}}{t^a}=0$ for $a\in \{0,1,2\}$, and this follows from $c_{(1,1,j)}(P)=0$.
    If we look at partial derivatives $\partial_{e'}$ evaluated at the point $(P,Q)=\pi_P([q_{ij}]_{i,j\in[n]})$, this is given by
    \begin{equation*}
    12t^2(e_{(0,j,j)})^2-4(3te_{(1,j,j)})^2.
    \end{equation*}
    We see that the partial derivative $\frac{\partial_{e}}{t^2}$ is given by 
    \[ 12(e^2_{(0,j,j)}-3e^2_{(1,j,j)}).\]
    For $j=2$ then this becomes
    \[27(4(q_{02}-q_{00}-q_{22})^2-3(q_{02}-q_{00}-q_{22})^2)=27(q_{02}-q_{00}-q_{22})^2\]
    and the claim follows from this quantity being non-zero at the point $(P,Q)$. 
    
    Suppose instead $q_{02}-q_{00}-q_{22}=0$, then we can look at different line conditions assuming ${e}_{(0,2,2)}(P,Q)={e}_{(1,2,2)}(P,Q)=0$. 
    Take $L_4^{j,t}$ where $j\neq 2$. If we repeat the same reasoning everything remains the same but in the end we get that $\frac{\partial_{e'}}{t^2}$ is given by
    \[27(4q_{jj}^2-3q_{jj}^2)=27q_{jj}^2.\]
    Once again, we obtain the claim if this quantity is different from zero for our $(P,Q)$. If instead $c_{jj}=0$ for every $j\neq 2$, then we can look at different line conditions assuming $e_{(0,j,j)}(P,Q)=e_{(1,j,j)}(P,Q)=0$ for every $j$. 
    Let us denote with $L^{i,j,t}$ the line conditions associated to the lines $\mathcal{V}(x_1+tx_j,x_2,\dots,\hat{x}_j,\dots,\hat{x}_i,x_i+x_j, \dots,x_n)$ for $i,j\neq 0,1$. The proper transform $L^{i,j,t}_3$ of these line conditions in the chosen affine chart for $V_3$ are given by
    \[4F_{i,j,t}^3+G_{i,j,t}^2 d_b=0\] where \[F_{i,j,t}=t^2+b_{(0,j,j)}+b_{(0,i,i)}-b_{(0,i,j)}-tb_{(0,1,j)}+tb_{(0,1,i)}\] and 
    \begin{align*}
    G_{i,j,t}=&\ t^3+c_{(i,i,i)}+3tc_{(1,i,i)}+3t^2c_{(1,1,i)}+3c_{(i,j,j)}\\&-3c_{(i,i,j)}+3tc_{(1,j,j)}-3t^2c_{(1,1,j)}-3tc_{(1,i,j)}-c_{(j,j,j)}.
    \end{align*}
    If we now consider the proper transform of this line condition in a chart of $V_4$ different from ${e}_d\neq 0$, repeating a similar reasoning to before we can see that for partial derivatives $\partial_{y}$ with respect to variables $y\neq e'$ evaluated at the point $(P,Q)$, we have $\frac{\partial_{y}}{t^a}=0$ for $a\in \{0,1,2\}$, and this follows again from the fact that $c_{(1,1,i)}(P)=c_{(1,1,j)}(P)=0$ for our point $P$.
    If we look at partial derivatives $\partial_{e'}$ for the variable $e'$ evaluated at the point $(P,Q)=\pi_P([q_{ij}]_{i,j\in[n]})$, this is given by
    \begin{equation*}
    12t^2(e_{(0,i,j)})^2-4(3te_{(1,i,j)})^2.
    \end{equation*}
    But then we see that $\frac{\partial_{e'}}{t^2}$ is given by 
    \[ 12(e^2_{(0,i,j)}-3e^2_{(1,i,j)}).\]
    If $j=2$ then this becomes
    \[27(4(q_{0i}-q_{2i})^2-3(q_{0i}-q_{2i})^2)=27(q_{0i}-q_{2i})^2\]
    and we obtain the claim if this quantity is different from zero for our $(P,Q)$. If instead we also have $q_{0i}-q_{2i}=0$ for every $i$, then we can look at different line conditions. 
    Take $L^{i,j,t}$ where $j\neq i\neq \{0,1,2\}$. If we repeat the same reasoning everything remains the same but in the end we get that $\frac{\partial_{e'}}{t^2}$ is given by
    \[27(4q_{ij}^2-3q_{ij}^2)=27q_{ij}^2.\]
     Once again, we obtain the claim if this quantity is different from zero for our $(P,Q)$. Finally, if  $q_{ij}=0$ for every $ij$ as before, then this implies $[q_{ij}]_{i,j\in[n]}\in (TB_3)_{P}$, but this is not possible.
    This concludes the proof.
    \end{proof}
    
    \begin{cor}\label{cor:empty}
    The intersection of all line conditions in $V_5$ is empty.
    \end{cor}
    
    \begin{proof}
    We need to show that the line conditions do not intersect in $E_5$. Thanks to Remark~\ref{rmk:invlinecond} and the fact that the equations in Remark \ref{rmk:equB3} do not involve the variable $c_a$, we can show it over fibers corresponding to points in $B_4\setminus (\pi_4|_{B_4})^{-1}(e)$. By the $\PGL_n$-action we can just look at one single fiber on a point $\overline{\lambda\mu^2}$ of $B_4$, with $\pi_4(\overline{\lambda\mu^2})=[\lambda\mu^2]$. The claim then follows from Lemma~\ref{lem:lineCondE4}.
    \end{proof}
    
    The previous lemma proves that line conditions separate in $V_5$ and that this space is a \textit{$1$--complete space of cubic hypersurfaces}.
    
    \subsection{Identifying the vector bundle $\mathcal{E}$ on $e$}
    We now give a more explicit description of the bundle $\mathcal{E}|_e$, which will be useful for understanding the total Chern class $c(\mathcal{E})$.
    
    \begin{lemma}\label{lemma:PGL_transitive_on_e}
        The natural action of $\PGL_n$ on the exceptional divisor $e \subseteq \Bl_{\Delta} \PP^n \times \PP^n$ is transitive.
    \end{lemma}
    
    \begin{proof}
    We have $e = \PP(N_{\Delta}\PP^n \times \PP^n) = \PP(T\Delta)$ where the isomorphism $N_{\Delta}\PP^n \times \PP^n \cong T\Delta$ is provided by any multiple of the difference of the differentials of the projections, e.g. $d\mathrm{pr}_1 - d\mathrm{pr}_2$. Fix now two points $[\lambda], [\mu] \in \Delta$ and two non-zero normal vectors $(v_1, v_2) \in N_{\Delta}\PP^n \times \PP^n|_{[\lambda]}$ and $(w_1, w_2) \in N_{\Delta}\PP^n \times \PP^n|_{[\mu]}$. These two normal vectors are represented by two curves $\mathbb{A}^1 \rightarrow \PP^n \times \PP^n, t \mapsto ([\lambda + t v_1],[\lambda + t v_2])$ and $t \mapsto ([\mu + t w_1],[\mu + t w_2])$, respectively. We then only need to find $A \in \PGL_n = \GL_{n+1}/\sim$ with $A\lambda = \mu$ and $A(v_1 - v_2) = w_1 - w_2$. Such a $A$ exists if $v_1 - v_2$ is not a multiple of $\lambda$ and $w_1 - w_2$ is not a multiple of $\mu$. Both conditions are satisfied by the requirement that $(v_1, v_2)$ and $(w_1, w_2)$ are both non-zero normal vectors.
    \end{proof}
    
    \begin{prop}\label{prop:identify_E}
	    We have $\mathcal{E}|_e \cong \Sym^2(T_{e/\Delta})$.
	\end{prop}
	
	\begin{rmk}
	  The geometric intuition behind this proposition is as follows. The fiber of $\Sym^2(T_{e/\Delta})$ over a point $([\lambda],[g]) \in e$ is $\frac{\Sym^2(W/\lambda)}{g \cdot (W/\lambda)}$, the \emph{quadrics on $g$}. This makes much sense, given that over a point $[\lambda \mu^2] \in B_3 \setminus e$, the fiber of $\mathcal{E}$ is naturally identified with $\frac{\Sym^2(W)}{(\lambda \cdot W + \mu \cdot W)}$, the \emph{quadrics on $\lambda \cap \mu$}. Fixing $\lambda$, as $\mu$ approaches $\lambda$ along some curve, $\lambda \cap \mu$ can be seen as a sequence of hyperplanes inside $\lambda$ with some limiting hyperplane $g$ inside $\lambda$. Along this sequence, the quadrics on $\lambda \cap \mu$ should indeed approach the quadrics on $g$.
	\end{rmk}
	
	\begin{rmk}\label{rem:chernEbundle}
	    It follows from the relative Euler sequence of the projective bundle $e$ over $\Delta$ that
	    \begin{equation*}
	        \Sym^2(T_{e/\Delta}) \cong \frac{\pi_e^\ast \left( \Sym^2(T\Delta)\right) \otimes \mathcal{O}_e(2)}{\pi_e^\ast \left( T\Delta \right) \otimes \mathcal{O}_e(1)},
	    \end{equation*}
	    the total Chern class of which can be computed using the Chern classes of $T\Delta$.
	\end{rmk}
	
	\begin{proof}[Proof of Proposition~\ref{prop:identify_E}]
	From the relative Euler sequence for the projective bundles $e = \PP(T\Delta)$ and $B_1 = \PP(\Sym^2(T\Delta))$ we obtain
	\begin{equation*}
	N_e B_1 \cong \frac{T_{B_1/B_0}|_e}{T_{e/\Delta}} \cong \Sym^2(T_{e/\Delta}).
	\end{equation*}
	We now first give an embedding of $N_e B_1$ into $N_{B_3} V_3|_e$.
	In a second step we show that the image agrees with $\mathcal{E}|_e$. For the first step, we observe the chain of natural inclusions of geometric vector bundles
	\begin{equation}\label{eq:chain_inclusions}
	    T_{B_2/B_0}|_e \subseteq TB_2|_e \cong T\PP(N_{B_2} E_2)|_e \subseteq T\PP(N_{B_2} V_2)|_e = TE_3|_e \subseteq TV_3|_e,
	\end{equation}
	using in the first step that $B_2 = E_2 \cap \tilde{E_1}$, so $N_{B_2} E_2$ is a line bundle and therefore the restriction of $\pi_3$ is an isomorphism $\PP(N_{B_2} E_2) \cong B_2$. In order to embed $T_{B_1/B_0}|_e$ into $T_{B_2/B_0}|_e$, note that $B_2$ is actually a fiber product over $B_0$. To be precise, it follows from Lemma~\ref{lem:iden} that
	\begin{equation*}
	    B_2 = \PP(N_{B_1} E_1) \cong \PP(p_1^\ast(\Sym^3(T\Delta))) = \PP(\Sym^3(T\Delta)) \times_{B_0} B_1.
	\end{equation*}
	The restriction $p_2: B_2 \rightarrow B_1$ of $\pi_2$ agrees under this identification with the projection to the second factor. Under the natural identifications $B_0 = \Delta$ and $B_1 = \PP(\Sym^2(T\Delta))$, the inclusion $e \subseteq B_2$ corresponds to the map
	\begin{equation*}
	    e = \PP(T\Delta) \overset{(\nu_3,\nu_2)}{\longrightarrow} \PP(\Sym^3(T\Delta)) \times_{\Delta} \PP(\Sym^2(T\Delta)),
	\end{equation*}
	where $\nu_2$, $\nu_3$ denote the relative second and third Veronese embeddings. On the fiber over $[\lambda] \in \Delta = \PP(W)$, these map a linear form $[g] \in \PP(W/\lambda) = e|_{[\lambda]}$ to its second respectively third power. Consider now the following diagram (where we omit the pullback signs and identify $B_0 = \Delta$):
	\begin{equation*}
	    \begin{tikzcd}
	        &
	            &
	                & 0 \ar[d]
	                    & \\
	        & T_{B_2/B_0} \ar[rr, dashed] \ar[rd, hook] \ar[dd, dashed]
	            &
	                & T_{B_1/B_0} \ar[d]
	                    & \\
	       &
	            & TB_2 \arrow[dr, phantom, "\square"] \ar[r] \ar[d]
	                & TB_1 \ar[d]
	                    & \\
	   0 \ar[r]
	       & T_{\PP(\Sym^3(T\Delta)) / \Delta} \ar[r]
	            & T\PP(\Sym^3(T \Delta)) \ar[r]
	                & T\Delta = TB_0 \ar[r] \ar[d]
	                    & 0 \\
	       &
	            &
	                & 0
	                    &
	    \end{tikzcd}
	\end{equation*}
	The induced dashed maps provide an isomorphism $T_{B_2/B_0} \cong T_{\PP(\Sym^3(T\Delta))/\Delta} \oplus T_{B_1/B_0}$. We define the embedding
	\begin{center}
    	\begin{tikzcd}
    	    s: T_{B_1/B_0}|_e \arrow[r, hook] & T_{\PP(\Sym^3(T\Delta))/\Delta}|_e \oplus T_{B_1/B_0}|_e \cong T_{B_2/B_0}|_e
    	\end{tikzcd}
	\end{center}
	by prescribing it to be the identity on the second factor. On the first factor we define it via
	\begin{equation*}
	    \begin{tikzcd}
	    T_{B_1/B_0}|_e \ar[r, hook] \ar[d, "\cong"']
	        & T_{\PP(\Sym^3(T\Delta))/\Delta}|_e \ar[d, "\cong"] \\
	    \frac{\pi_e^\ast(\Sym^2(T\Delta)) \otimes \mathcal{O}_e(2)}{\mathcal{O}_e} \ar[r, hook]
	        & \frac{\pi_e^\ast(\Sym^3(T\Delta)) \otimes \mathcal{O}_e(3)}{\mathcal{O}_e},
	    \end{tikzcd}
	\end{equation*}
	given on the fiber over $([\lambda],[g]) \in e$ (i.e. $[\lambda] \in \Delta = \PP(W)$ and $g \in W/\lambda$) by sending a quadric $q \in \Sym^2(W/\lambda)/(g^2)$ to $\mathrm{cst} \cdot g \cdot q \in \Sym^3(W/\lambda)/(g^3)$, where $\mathrm{cst}$ is some non-zero constant still to be specified. (Up to multiplication by $\mathrm{cst}$, this is a relative version of the map $\alpha_{2,3}$ from Lemma~\ref{lemma:embedding_normal_bundles}.) Denoting by $e_1 \subseteq B_1$ and $e_2 \subseteq B_2$ the images of $\phi_1(e)$ and $\phi_2(e)$, respectively, it is important to observe that $s$ satisfies $s(T_{e_1/\Delta}) = T_{e_2/\Delta} \subseteq T_{B_2/B_0}|_{e_2}$. By composing with \eqref{eq:chain_inclusions}, the embedding $s: T_{B_1/B_0}|_e \hookrightarrow T_{B_2/B_0}|_e$ now provides an embedding of geometric vector bundles $T_{B_1/B_0}|_e \hookrightarrow TV_3|_e$. Composing further with the quotient map $TV_3|_e \to N_{B_3} V_3|_e$, the kernel is precisely $T_{e_1/B_0} \subseteq T_{B_1/B_0}|_{e_1}$. Hence, we obtain an embedding of geometric vector bundles
	\begin{center}
    	\begin{tikzcd}
    	    N_e B_1 \arrow[r, hook] & N_e E_3 \subseteq N_{B_3} V_3|_e.
    	\end{tikzcd}
	\end{center}
	We denote by $\mathcal{F} \subseteq N_e E_3 \subseteq N_{B_3} V_3|_e$ its image. It is enough to show $\PP(\mathcal{F}) = \PP(\mathcal{E}|_e) = B_4 \cap \pi_4^{-1}(e)$. As the embedding $N_e B_1 \hookrightarrow N_{B_3} V_3|_e$ is $\PGL_n$-equivariant, it is enough to show the equality $\PP(\mathcal{F}) = \PP(\mathcal{E}|_e)$ for the fiber over a single point of $e$, using that $\PGL_n$ acts transitively on $e$ by Lemma~\ref{lemma:PGL_transitive_on_e}. We pick the point $([\lambda], [g]) = ([x_0], [x_1]) \in e$. In the explicit coordinates of Subsection~\ref{subsec:fourblow}, the fiber of $B_4$ over this point is defined (in the affine chart where $e_{(0,1,2)} = 1$) by the equations
	\begin{alignat*}{2}
	    0 &= e', &&\\
	    0 &= e_d, &&\\
	    0 &= e_{(0,i,j)} - 2e_{(1,i,j)} \quad &&\text{for all } i,j>1, \\
	    0 &= e_F \quad &&\text{for all } F \neq (0,i,j), (1,i,j).
	\end{alignat*}
	This follows from the fact that the same equations hold for the fiber of $B_4$ over the point $[(x_0 + x_1 + x_2)(x_0 + x_2)^2] \in B_3 \setminus e$, see the proof of Lemma~\ref{lem:lineCondE4}. This point has the same $b_{(0,i,j)}$ and $c_{(i,j,k)}$ coordinates as $([x_0], [x_1]) \in e$. By Remark~\ref{rmk:equB3}, the equations for $B_3$ inside $V_3$ only depend on those, and the same is true for the equations of the line conditions from Remark~\ref{rmk:invlinecond}. Therefore, the fiber of $B_4$ over $([x_0], [x_1]) \in e$ is indeed defined by the same equations in the $(e_F, e_d, e')$-coordinates as the fiber over $[(x_0 + x_1 + x_2)(x_0 + x_2)^2] \in B_3 \setminus e$. Finally, the explicit description of the embedding $s: T_{B_1/B_0}|_e \hookrightarrow T_{B_2/B_0}|_e$ above provides a way to check that all points in the chart satisfying the above equations lie inside $\PP(\mathcal{F}|_{([x_0],[x_1])})$. Namely, if we start with a tangent vector associated to a quadric $q = \sum_{2 \leq i \leq j} q_{(i,j)} x_i x_j \in T_{B_1/B_0}|_{([x_0],[x_1])}$, it is represented by a curve in our usual affine open chart of $B_2$ given by sending $t \in \mathbb{A}^1$ to $b_{(0,i,j)} = q_{(i,j)} \cdot t$, $c_{(1,i,j)} = \frac{\mathrm{cst}}{3} \cdot q_{(i,j)} \cdot t$ for all $i,j>1$ and all other coordinates equal to $0$. Tracing the proper transform of this curve in $V_3$, we obtain that this tangent vector corresponds to the point in $E_4$ with coordinates $e' = e_d = 0$, $e_{(0,i,j)} = q_{(i,j)}, e_{(1,i,j)} = \frac{\mathrm{cst}}{3} \cdot q_{(i,j)}$ and all other $e_F = 0$. This satisfies the above equations exactly for the choice $\mathrm{cst} = \frac{3}{2}$. We get that $\PP(\mathcal{F}|_{([x_0],[x_1])})$ contains a dense open subset of $\PP(\mathcal{E}|_{([x_0],[x_1])})$ and hence the entire fiber. As their dimensions agree, we obtain equality, and with this we conclude that $\PP(\mathcal{F}) = \PP(\mathcal{E}|_e)$.
	\end{proof}

	\section{Intersection rings and Chern classes}\label{sec:3}
	
	In the following subsections we collect details about the Chow rings of the centers of the blow-ups. This is the last step needed to compute the characteristic numbers. In particular, we find generators, describe the degree of the product of those generators, and find the Chern classes $c(N_{B_i}V_i)$ of the normal bundle of $B_i$ inside $V_i$.
	Finally, we compute the \textit{full intersection classes} $B_i\circ P_i$ and $B_i\circ L_i$, which are defined in \cite[Section 2]{Aluffi:1990}.
	\begin{rmk}\label{rmk:intersClasses}
	For $X\subseteq V_i$ being a divisor and $j_i:B_i\hookrightarrow{} V_i$, we have \begin{equation}
	     B_i \circ X = e_{B_i}X[B_i] + j_i^*[X]
	\end{equation}
	where $e_{B_i}X$ denotes the multiplicity of $X$ along $B_i$.
	\end{rmk}
	\setcounter{subsection}{-1}
	\subsection{Chow ring of $\boldsymbol{B_0}$}
	The following results directly generalize from \cite{Aluffi:1990}:
	\begin{lemma} The intersection ring of $B_0\simeq\PP^n$ is generated by the hyperplane class $h$.  Moreover, $c(N_{B_0}V_0)=(1+3h)^{\binom{n+3}{3}}/(1+h)^{n+1}$.
	\end{lemma}
	\begin{lemma} The full intersection classes of point and line conditions in $V_0$ with respect to $B_0$ are
		$$B_0\circ P=3h \qquad
		B_0\circ L=2+12h$$
	\end{lemma}
	
	\subsection{Chow ring of $\boldsymbol{B_1}$} 
	The center of the second blow-up is $B_1$, this was described in subsection \ref{subsec:oneblow}.
	\begin{lemma}\label{lem:onechow} The variety $B_1$ has dimension $\binom{n+2}{2}-2$.
	\begin{enumerate}
    \item The intersection ring of $B_1$ is generated by the pullback $h$ of $h$ via the map $\pi_1:B_1\to B_0$ and the pullback $\epsilon$ of $[E_1]$ via the inclusion map $j_1:B_1\hookrightarrow V_1$. Consider the sequence $\{a_s\}$ obtained with the following recursion     \[\begin{cases} a_0=1\\a_s=\binom{n+1}{s}-\displaystyle\sum_{i=1}^{s}2^i\binom{\dim B_1+2}{i}a_{s-i}.\end{cases}\]
     We get
    \begin{alignat*}{2}
    \int_{B_1} h^{s}\epsilon^{\dim B_1-s}&=0
    \qquad&& \hbox{ for } s\in [n+1,\dim B_1],\medskip\\
    \int_{B_1} h^{s}\epsilon^{\dim B_1-s}&=(-1)^{\dim B_1-s}a_{n-s}
    \qquad&& \hbox{ for } s\in[0,n].
    \end{alignat*}

    \item $c(N_{B_1}V_1)=(1+\epsilon)(1+3h-\epsilon)^{\binom{n+3}{3}}/(1+2h-\epsilon)^{\binom{n+2}{2}}$
    \end{enumerate}
    \end{lemma}
    \begin{proof} The proof is exactly as in \cite{Aluffi:1990}. Notice that $\mathcal{O}_{E_1}(-1)=\mathcal{O}_{V_1}(E_1)_{|_{E_1}}$, so that $c_1(\mathcal{O}_{E_1}(-1))=c_1(\mathcal{O}_{V_1}(E_1)_{|_{E_1}})=j_1^*([E_1])$.
    The coefficients in
    \[
    s(N_{v_2(\PP^n)}\PP^{\binom{n+2}{2}-1})=\frac{(1+h)^{n+1}}{(1+2h)^{\binom{n+2}{2}}} = a_0 + a_1 h + \cdots + a_n h^n
    \]
    are computed by equating the coefficients of the powers of $h$.
    From the above expression one gets the relation
    \[
    \binom{n+1}{s} = \sum_{i = 0}^s 2^{s-i}\binom{\binom{n+2}{2}}{s-i}a_i
    \]
    from which we attain the recursive formula for the $a_i$'s stated above.
    \end{proof}

    \begin{rmk}[$n=3$]
    In this case we have:
    \begin{equation*}
        s(N_{v_2(\PP^3)}\PP^9) = 1 - 16 h + 146 h^2 - 996 h^3.
    \end{equation*}
    The center $B_1$ has dimension $8$ and
    $\int_{B_1}h^j \epsilon^{8-j}  = 0 \text{ when } j \geq 4$, while the other intersection numbers are summarized in Table \ref{table:intsurfaces-1},
    \begin{table}[ht]
    \begin{center}
    \begin{tabular}{|c|c|c|c|} 
    \hline
    1 & $h$ & $h^2$ & $h^3$  \\
    \hline
    $-996$ & $-146$ & $-16$ & $-1$\\
    \hline
    \end{tabular} 
    \vspace{0.2cm}
     \caption{The intersection number of $\int_{B_1}h^j \epsilon^{8-j}$ is given in column $l^j$.}
    \label{table:intsurfaces-1}
    \end{center}
    \end{table}
    \end{rmk}
    
    \begin{lemma} We have $\pi_1^*(P) = P_1$ and $\pi_1^*(L) = L_1 + 2 E_1$.
    The full intersection classes of point and line-conditions with respect to $B_1$ are: 
    \begin{align*}
        B_1 \circ P_1 & = 3h, & B_1 \circ L_1 & = 1 + 12 h - 2\epsilon.\\
    \end{align*}
    \end{lemma}
    
    \subsection{Chow ring of $\boldsymbol{B_2}$} 
    The third center of blow-up is $B_2$, which was described in the subsection \ref{subsec:twoblow}.
    \begin{lemma}
    $B_2$ is a $\PP^{\binom{n+2}{3}-1}$-bundle over $B_1$ and it has dimension $\binom{n+3}{3}-3$.
    \begin{enumerate}
        \item Consider $\pi_{2_{|_{B_2}}}: B_2 \to B_1$ and the inclusion $j_2: B_2 \hookrightarrow V_2$.
        Then the intersection ring of $B_2$ is generated by the pullback of the classes $h$ and $\epsilon$ along the projection $\pi_2 $ and by $\phi$ the pullback of $[E_2]$ along the inclusion $j_2$. Let $\{c_{j,k}\}_{j\in[n],k\in [\dim B_1]}$ be the sequence obtained recursively for $j + k \leq \dim B_1$, following the lexicographic order:
    \[\begin{cases}
    c_{0,0}=1 \\
              c_{j,k}=(-1)^k 2^j \binom{\dim B_1+2}{j,k}-\displaystyle\sum_{\substack{a\leq j, b\leq k \\ (a,b) \neq (j,k)}} c_{a,b}\binom{\dim B_2+3}{j-a,k-b}3^{j-a}(-1)^{k-b},
    \end{cases}\]where $\binom{a}{b,c}=\frac{a!}{(a-b-c)! b!c!}$. For $j\in[n], k\in [\dim B_1]$, we get
    \[\int_{B_2} h^{j}\epsilon^k\phi^{\dim B_2-j-k}=0 \qquad\qquad \]
   for $j+k\in [\dim B_1+1,\dim B_2]$ and
    \[\int_{B_2} h^{j}\epsilon^{k}\phi^{\dim B_2-j-k}=(-1)^{\dim B_2-j-k}\sum_{\substack{a+b=\dim B_1-j-k \\ a\in[n-j],b\in[\dim B_1]}}c_{a,b}(-1)^{\dim{B_1}-a-j}a_{n-a-j}\]
    for $j+k\in[0,\dim B_1]$, where the $a_{i}$ in the sum are the numbers obtained by recursion in Lemma \ref{lem:onechow}.
        \item Moreover $c (N_{B_2}V_2) = (1 + \phi)(1 + \epsilon - \phi)$
    \end{enumerate}
    \end{lemma}
    \begin{proof}
    Since $B_2=\PP(N_{B_1}V_1)$, the first point follows from \cite[Example 8.3.4]{fulton:intersection}.
    Point (ii) in Lemma \ref{lem:onechow} implies
    \[
    c(N_{B_1}E_1) = \frac{(1 + 3h - \epsilon)^{\binom{n+3}{3}}}{(1 + 2h - \epsilon)^{\binom{n+2}{2}}}.
    \]
    Hence recalling that $h^{n+1} = 0$ and $\epsilon^{\binom{n+2}{2}} = 0$, we have
    \[
    s(N_{B_1}E_1) = \frac{(1 + 2h - \epsilon)^{\binom{n+2}{2}}}{(1 + 3h - \epsilon)^{\binom{n+3}{3}}} = \sum_{\substack{j \in [n],k \in [\binom{n+2}{2}]}} c_{j,k} h^j \epsilon^k.
    \]
    Note that this relation allows to give a recursive formula for the coefficients $c_{j,k}$ (following the lexicographic order on $(j,k)$).
    For the second point we have that $B_2 = \tilde{E_1} \cap E_2$ hence $c(N_{B_2}V_2) = c(N_{E_2}V_2)c(N_{\tilde{E_1}}V_2)$.
    Note that $N_{E_2}V_2$ is a line bundle on $E_2$ and $N_{\tilde{E_1}}V_2$ is a line bundle on $E_1$.
    Since $N_{E_2}V_2=\mathcal{O}_{E_2}(-1)$, we get $c(N_{E_2}V_2)=1+\phi$. The isomorphism  $N_{\tilde{E_1}}V_2\simeq \pi_1^*(N_{E_1}V_1)\otimes \mathcal{O}_{E_2}(-E_2)$ implies $c(N_{\tilde{E_1}}V_2)=1+\epsilon-\phi$.
    Hence, one concludes that
    \[c(N_{B_2}V_2) = c(N_{E_2}V_2)c(N_{\tilde{E_1}}V_2) = (1 + \phi)(1 + \epsilon - \phi).\]
    \end{proof}
    
    \begin{rmk}[$n=3$]
    In the case of cubic surfaces, $B_2$ has dimension $17$ and
    $\int_{B_2}h^j\epsilon^k \phi^{17-j-k}  = 0 \text{ with } j+k > 8$ or $j > 3$, while the remaining integrals are summarized in Table \ref{table:intsurfaces-2}.
    \begin{table}[ht]
    \begin{center}
    \begin{tabular}{|c|c|c|c|c|c|c|c|c|c|} 
    \hline
    & 1 & $\epsilon$ & $\epsilon^2$ & $\epsilon^3$ & $\epsilon^4$ & $\epsilon^5$ & $\epsilon^6$ & $\epsilon^7$ & $\epsilon^8$  \\
    \hline
    1 & -1370200 & -641680 & 251160 & 24388 & -49400 & 12900 & 4460 & -4120 & 996\\
    \hline
    $h$  & -345280 & -3640 & 31668 & -10790 & -320 & 1860 & -820 & 146 & 0 \\
    \hline
    $h^2$ & -40040 & 8008 & 0 & -880 & 440 & -120 & 16 & 0 & 0 \\
    \hline
    $h^3$ & -2002 & 715 & -220 & 55 & -10 & 1 & 0 & 0 & 0 \\
    \hline
    \end{tabular}
    \end{center}
    \caption{The intersection number of $\int_{B_2}h^j\epsilon^k \phi^{17-j-k}$ is given in row $h^j$ and column $\epsilon^k$.}
    \label{table:intsurfaces-2}
    \end{table}
    \end{rmk}
    
    \begin{lemma}We have $\pi_2^*(P_1) = P_2$ and $\pi_2^*(L_1) = L_2 +  E_2$.
    The full intersection classes of point and line conditions with respect to $B_2$ are: 
    \begin{align*}
        B_2 \circ P_2 & = 3h, & B_2 \circ L_2 & = 1 + 12 h - 2\epsilon-\phi.\\
    \end{align*}
    \end{lemma}
    
    \begin{proof}
    As divisor classes we have $\pi_2^*(P_1) = P_2$ and $\pi_2^*(L_1) = L_2 +  E_2$, because $L_1$ is generically smooth along $B_1$.
    In the intersection ring of $B_2$ we get $j_2^*(P_2) = 3h$ and $j_2^*(L_2) = 12h - 2\epsilon-\phi$.
    The claim follows by Remark~\ref{rmk:intersClasses} observing that in our case the divisor $P_2$ does not contain $B_2$ and $L_2$ is smooth along $B_2$.
    \end{proof}
    
    \subsection{Chow ring of $\boldsymbol{B_3}$} The fourth center we blow-up is $B_3$, which was described in subsection \ref{subsec:thirdblow}. Recall that it is defined to be the proper transform in $V_3$ of $S_0$ and lemma \ref{lemma:phi3} we described an isomorphism $\phi_3$ of $S_3$ with $\Bl_{\Delta}\PP^n\times\PP^n$.
    
    \begin{thm}We identify $B_3$ with $\Bl_{\Delta}\PP^n\times\PP^n$.
    \begin{enumerate} 
        \item We have two natural projections of $\Bl_{\Delta}\PP^n\times\PP^n$ onto $\PP^n$.  Let $l,m$ be the pullbacks of the hyperplane classes in $\PP^n$ through these projections, and let $e$ denote the exceptional divisor. Consider the sequence $\{d_s\}$ obtained recursively:     \[\begin{cases} d_0=1\\d_s=-\sum_{i=0}^{s-1}\binom{n+1}{s-i}d_{i}.\end{cases}\]The intersection ring of $B_3$ is generated by $l,m,e$ subject to $em^s=el^s$ for every $s$, $l^{n+1}=m^{n+1}=0$ and  
              \[\int_{B_3} m^{n}l^{n}=1, \qquad \qquad \int_{B_3} l^{s}e^{2n-s}=(-1)^{2n-s-1}d_{n-s} \]with $s\in [n]$. All the remaining intersection numbers vanish.
        \item ($n=3$) In the case of cubic surfaces, we have 
        {\small
        \setlength{\jot}{0.9pt}
            \begin{align*}
                c(N_{B_3}V_3)=&1672560l^3m^3 - 66343820m^3e^3 + 36537350m^2e^4 - 10851224m e^5 + 1356403e^6 +\\ &209440 l^3 m^2 + 474320 l^2 m^3 + 8045100 m^3 e^2 - 5907690 m^2 e^3 + 2193180 m e^4 - \\ &328977 e^5 + 15960 l^3 m + 53560 l^2 m^2 + 81680 l m^3 - 582940 m^3 e + 642110 m^2 e^2 -\\ &317840 m e^3 + 59595 e^4 + 560 l^3 + 3720 l^2 m + 8400 l m^2 + 6460m^3 - 42166 m^2 e + \\& 31308 m e^2 - 7827 e^3 + 120 l^2 + 536 l m + 610 m^2 - 1880 m e + 705 e^2 + 16 l + 36 m -\\& 39 e + 1.
            \end{align*}}
    \end{enumerate}
    \end{thm}
    
    \begin{proof}
    We follow the proof of \cite[Theorem III (4)]{Aluffi:1990}. Consider the diagram:
    \[ 
        \begin{tikzcd}
        {e} \arrow[r, "j"] \arrow[d,"g"]
            & {\Bl_{\Delta}\PP^n\times\PP^n} \arrow[d, "f"] \\
        {\Delta} \arrow[r, "i"']
            & {\PP^n\times\PP^n}.
        \end{tikzcd}
    \]
    By definition we have $s(\Delta,\PP^n\times\PP^n)=s(N_{\Delta}(\PP^n\times\PP^n))$ and 
    \[s_i(N_{\Delta}(\PP^n\times\PP^n))=g_{*}((-\zeta)^{n-1+i})\]
    where $\zeta=c_1(\mathcal{O}_{\PP(e)}(-1))=j^*(e)$.
    Hence we get 
    \[s(\Delta,\PP^n\times\PP^n)=g_{*}\biggl(\sum_{i=0}^{n}(-1)^{n-1+i}(j^*e)^{n-1+i}\biggr).\]
    However, we can also compute $s(\Delta,\PP^n\times\PP^n)$ as the inverse of the Chern class of $N_{\Delta}(\PP^n\times\PP^n)$. This gives us
    \[s(\Delta,\PP^n\times\PP^n)=\frac{1}{(1+k)^{n+1}}=\left(\sum_{i=0}^{n}(-1)^ik^i\right )^{n+1}.\]
    Applying the projection formula and the above observations we get
    \begin{align*}
    \int_{B_3}f^*(l)^ie^{2n-i}&=\int_{B_3}j_{*}(j^*(f^*l^i)j^*(e^{2n-i-1}))=\int_{\Delta}k^ig_{*}(j^*(e)^{2n-i-1})=\\
    &=\int_{\Delta}k^i(-1)^{2n-i-1}s(\Delta,\PP^n\times\PP^n)=(-1)^{2n-i-1}d_{n-i}.
    \end{align*}
    For proving point (ii), we can compute $\phi_3^*(c(TV_3))$ applying \cite[Theorem 15.4]{fulton:intersection} multiple times. Notice that to do this, computing $c(TB_1)$ and $c(TB_2)$ is also needed. Hence we observe that \[c(TB_1)=\frac{(i_1\circ j_1)^*c(TV_1)}{c(N_{B_1}V_1)}\] and then we can  again use \cite[Theorem 15.4]{fulton:intersection} to compute the numerator. 
    Finally, $c(TB_3)$ is computed applying \cite[Theorem 15.4]{fulton:intersection} considering $B_3=\Bl_{\Delta}\PP^n\times\PP^n$.
    \end{proof}
    
    \begin{rmk}
    Notice that in principle it is possible to compute $c(N_{B_3}V_3)$ for any $n$, using the strategy of the proof. However, this can be computationally difficult. For instance, here is the result for $n=4$:
    {\footnotesize
        \setlength{\jot}{0.9pt}
    \begin{align*}
    c(N_{B_3}V_3) = \ &8604607900l^4m^4 + 1511859296400m^4e^4 - 956335227000m^3e^5 + 379626653775m^2e^6- \\& 86448428700me^7 + 8644842870e^8 + 699244875l^4m^3 + 1520696100l^3m^4 - \\& 107772730500m^4e^3 + 85215404025m^3e^4 - 40592536260m^2e^5 + 10784338950me^6 - \\&1232495880e^7 + 40828725l^4m^2 + 117863200l^3m^3 + 192910550l^2m^4 + 5484228225m^4e^2 - \\& 5781808210m^3e^3 + 3442721815m^2e^4 - 1097565900m e^5 + 146342120e^6 +  1525545l^4 m + \\& 6578880 l^3 m^2 + 14291235 l^2 m^3 + 15643810 l m^4 - 177497950 m^4 e + 280693735 m^3 e^2 - \\& 222848500 m^2 e^3 + 88807250 m e^4 - 14209160 e^5 + 27405 l^4 + 235480 l^3 m + 764065 l^2 m^2 + \\& 1109920 l m^3 + 609280 m^4 - 8685470 m^3 e + 10343355 m^2 e^2 - 5495900 m e^3 + 1099180 e^4 + \\& 4060 l^3 + 26245 l^2 m + 56940 l m^2 + 41475 m^3 - 306580 m^2 e + 244350 m e^2 - 65160 e^3 +\\& 435 l^2 + 1880 l m + 2045 m^2 - 6950 m e + 2780 e^2 + 30 l + 65 m - 76 e + 1
    \end{align*}}
    \end{rmk}
    
    \begin{rmk}[$n=3$]
    In this case the dimension of $B_3$ is $6$ and
    $\int_{B_3}l^jm^k e^{6-j-k}  = 0 \text{ with } j > 3$  or $k>3$, $\int_{B_3}l^3 m^{3}  = 1$ while the other integrals are summarised in Table \ref{table:intsurfaces-3}, remembering that $\int_{B_3}l^j m^k e^{6-j-k}=\int_{B_3}l^{j+k} e^{6-j-k}$ for $j+k<6$.
    \end{rmk}
    
    \vspace{-35pt}
    \begin{table}[ht]
    \begin{tabular}{|c|c|c|c|}
    \hline
    1 & $l$ & $l^2$ & $l^3$  \\
    \hline
    $20$ & $10$ & $4$ & $1$\\
    \hline
    \end{tabular} 
    \caption{The intersection number of $\int_{B_3}l^j e^{6-j-k}$ is given in column $l^j$.}
    \label{table:intsurfaces-3}
    \end{table}
    \vspace{-35pt}
    
    \begin{lemma}
    We have $\pi_3^*(P_2) = P_3$ and $\pi_3^*(L_2) = L_3 +  E_3$.
    The full intersection classes of point and line-conditions with respect to $B_3$ are: 
    \begin{align*}
        B_3 \circ P_3 & = l+2m, & B_3 \circ L_3 & = 1+4l+8m-6e.
    \end{align*}
    \end{lemma}
    
    \begin{proof}
    In the intersection ring of $B_3$ we get $j_3^*(P_3) =l+2m$ and $j_3^*(L_3) = 4(l+2m) - 2(2e)-e-e=4l+8m-6e$.
    The assertion is then proved by noticing that $P_3$ does not contain $B_3$ and $L_3$ is smooth along $B_3$.
    \end{proof}
    
    \subsection{Chow ring of $\boldsymbol{B_4}$} The fifth center we blow up is $B_4$, which was described in subsection \ref{subsec:fourblow}. In particular, recall that $B_4=\PP(\mathcal{E})$ and that we have an isomorphism of $\mathcal{E}|_e$ given in Proposition \ref{prop:identify_E}.
    
    \begin{thm}\label{thm:B4}
    We identify $B_4$ with $\PP(\mathcal{E})$.
    \begin{enumerate} 
        \item Let $l,m,e$ be the pullbacks of the generators of the Chow ring of $B_3$ through the projection $\pi_4|_{B_4}:B_4 \to B_3$. The Chow ring of $B_4$ is generated by $l,m,e$ and $z$ where $z$ is the first Chern class of $\mathcal{O}_{B_4}(-1)$. The intesection numbers in the case $n=3$ are collected in Table \ref{table:intsurfaces-4}.
            \begin{table}[ht]
                \begin{center}
                \begin{tabular}{|c|c|c|c|c|c|c|c|c|c|} 
                    \hline
                    & $e$ & $e^2$ & $e^3$ & $e^4$ & $e^5$ & $e^6$   \\
                    \hline
                    1 & -1820  &  -580 &   340 &   12 &  -60 &    20\\
                    \hline
                    $m$  & -890  & 190 &    54 &  -42 &   10 &     0 \\
                    \hline
                    $m^2$ & 0   & 68 &  -24 &    4  &   0  &   0  \\
                    \hline
                    $m^3$ &  51 &   -9  &   1  &   0    & 0 &    0 \\
                    \hline
                \end{tabular}
                \end{center}
                
                \vspace*{0.2cm}
                
                \begin{center}
                \begin{tabular}{|c|c|c|c|c|} 
                    \hline
                    & $1$ & $m$ & $m^2$ & $m^3$  \\
                    \hline
                    $1$ & 13720 & 1610 &  -600 &  -175\\
                    \hline
                    $l$  & 1610 & -230 & -35 & 21 \\
                    \hline
                    $l^2$ & -600 & -35 & 46 & 6 \\
                    \hline
                    $l^3$ & -175 & 21 & 6 & 1 \\
                    \hline
                \end{tabular}
                \vspace{0.2cm}
                \caption{The intersection number of $\int_{B_4}m^j e^{k} z^{8-j-k}$ is given in row $m^j$ and column $e^k$, while $\int_{B_4}m^j l^{k} z^{8-j-k}$ is given in row $l^k$ and column $m^j$}
                \label{table:intsurfaces-4}
                \end{center}
            \end{table}
        \item \label{item:Chern_classes_E}($n=3$) In the case of cubic surfaces, we have 
        {\footnotesize
        \setlength{\jot}{0.9pt}
            \begin{align*}
                c(N_{B_4}V_4)=&-8540 e^6z^2 - 45500l^2m^2z^4 - 109900lm^3z^4 + 280350m^2e^2z^4 - 325500me^3z^4 \\ & + 106575e^4z^4 + 13440l^2mz^5 + 44800lm^2  z^5 + 47320m^3z^5 - 235200m^2ez^5 + 174720me^2z^5 \\&- 43680e^3z^5 - 1260l^2z^6- 7560lmz^6 - 11620m^2z^6 + 30240mez^6 - 11340e^2z^6 + 480lz^7 \\&+ 1440mz^7 - 1440ez^7 - 75z^8 + 12810e^6z + 251300l^2m^3z^2 + 195650me^4z^2 - 108220e^5z^2 \\&- 45500l^2m^2z^3 - 109900lm^3z^3 + 280350m^2e^2z^3 - 325500me^3z^3 + 106575e^4z^3 + 630l^2z^5\\& + 3780lmz^5 + 5810m^2z^5 - 15120mez^5 + 5670e^2z^5 - 420lz^6 - 1260mz^6 + 1260ez^6 + 90z^7 \\&- 4270e^6 - 201040l^2m^3z - 156520me^4z + 86576e^5z + 81900l^2m^2z^2 + 197820lm^3z^2 \\&- 504630m^2e^2z^2 + 585900me^3z^2 - 191835e^4z^2 - 13440l^2mz^3 - 44800lm^2z^3 - 47320m^3z^3\\& + 235200m^2ez^3 - 174720me^2z^3 + 43680e^3z^3 + 630l^2z^4 + 3780lmz^4 + 5810m^2z^4\\& - 15120mez^4 + 5670e^2z^4 - 42z^6 + 50260l^2m^3 + 39130me^4 - 21644e^5 - 45500l^2m^2z \\&- 109900lm^3z + 280350m^2e^2z - 325500me^3z + 106575e^4z + 13440l^2mz^2 + 44800lm^2z^2 \\&+ 47320m^3z^2 - 235200m^2ez^2 + 174720me^2z^2 - 43680e^3z^2 - 1260l^2z^3 - 7560lmz^3 \\&- 11620m^2z^3 + 30240mez^3 - 11340e^2z^3 + 420lz^4 + 1260mz^4 - 1260ez^4 - 42z^5 + 9100l^2m^2 \\&+ 21980lm^3 - 56070m^2e^2 + 65100me^3 - 21315e^4 - 5760l^2mz - 19200lm^2z - 20280m^3z \\&+ 100800m^2ez - 74880me^2z + 18720e^3z + 900l^2z^2 + 5400lmz^2 + 8300m^2z^2 - 21600mez^2 \\&+ 8100e^2z^2 - 480lz^3 - 1440mz^3 + 1440ez^3 + 90z^4 + 960l^2m + 3200lm^2 + 3380m^3 - 16800m^2e\\& + 12480me^2 - 3120e^3 - 315l^2z - 1890lmz - 2905m^2z + 7560mez - 2835e^2z + 270lz^2 + 810mz^2\\& - 810ez^2 - 75z^3 + 45l^2 + 270lm + 415m^2 - 1080me + 405e^2 - 80lz - 240mz + 240ez + 35z^2\\& + 10l + 30m - 30e - 9z + 1
            \end{align*}}
    \end{enumerate}
    \end{thm}
    
    \begin{proof}
    To compute the intersection numbers, we follow the usual strategy adopted in the previous proofs, using 
    \[s(\mathcal{E})=(\pi_4|_{B_4})_*\left(\sum_{i=0}^{\dim(B_4)}(-1)^{i}z^i\right).\]
    
    The Segre class $s(\mathcal{E})$ is computed explicitly in the case $n=3$ thanks to Remark \ref{rem:chernEbundle}.
    In fact, we have the Chern Classes $c(\mathcal{E}|_{B_3\setminus e})$ and $c(\mathcal{E}|_e)$. The first one is in the Chow ring of $B_3\setminus e$, which thanks to the excision theorem is described by \[\frac{\mathbb{Z}[l,m]}{(l^{n+1},m^{n+1},[\Delta])}\] where $[\Delta]=\sum_{i=0}^{n}l^{n-i} m^i$. Therefore $c(\mathcal{E}|_{B_3\setminus e})=\sum_{d=0}^{2n}\sum_{j=0}^{d}o_{j,d-j}m^jn^{d-j}$.
    
    The second one is an element of the Chow ring of $e$ described by \[\frac{\mathbb{Z}[k,\zeta]}{(k^{n+1},\zeta^{2n},(n+1)k^{n}+\sum_{i=1}^{n}\zeta^ik^{n-i}c_{n-i}(T\Delta))}\]
    where $\zeta$ is just the pull-back of the class $e$ through the inclusion of $e$ in $B_3$, and $c_i(T\Delta)$ is the $i$-th Chern class of the tangent space of the diagonal. Therefore $c(\mathcal{E}|_{B_3\setminus e})=\sum_{d=0}^{2n}\sum_{j=0}^{d}u_{j,d-j}k^j\zeta^{d-j}$.
    In general, knowing $c(\mathcal{E}|_{B_3\setminus e})$ and $c(\mathcal{E}|_e)$ allows us to reconstruct $c(\mathcal{E})$ up to degree $n$. In fact, for $d<n$ we have \[c_d(\mathcal{E})=\sum_{j=0}^{d}o_{j,d-j}m^jl^{d-j}+\sum_{j=0}^{d-1}u_{j,d-j}m^je^{d-j}.\]
    For $d=n$, call then $\sum_{j=0}^{n}o_{j,n-j}=u_{n,0}+(n+1){w}$ for some ${w}\in\mathbb{Z}$ and 
    \[c_n(\mathcal{E})=\sum_{j=0}^{n}o_{j,n-j}m^jn^{n-j}-w[\Delta]+\sum_{j=0}^{n-1}u_{j,n-j}k^j\zeta^{n-j}.\]
    When $n=3$, since rk$(\mathcal{E})=3$, we get the entire total Chern class.
    
    For proving \eqref{item:Chern_classes_E}, we can follow the strategy of \cite[Theorem III(ii)]{Aluffi:1990} and Lemma \ref{lemma:NB4V4iso}.
    \end{proof}
    
    \begin{rmk}\label{rmk:goeswrong}
    We are not able to recover the Chern class $c(\mathcal{E})$ in the case $n>3$. Indeed, knowing $c(\mathcal{E}|_e)$ and $c(\mathcal{E}|_{B_3\setminus e})$ we can recover the $c(\mathcal{E})$ only up to integer multiples of $m^k [\Delta] = m^{k}(\sum_{i=1}^{n}e^i m^{n-i} c_{n-i}(T\Delta)))$ for $k\geq 1$.
    \end{rmk}

    \begin{lemma}We have $\pi_4^*(P_3) = P_4$ and $\pi_4^*(L_3) = L_4 +  E_4$.
    The full intersection classes of point and line-conditions with respect to $B_3$ are: 
    \begin{align*}
        B_4 \circ P_4 & = l+2m, & B_4 \circ L_4 & = 1+4l+8m-6e-z.
    \end{align*}
    \end{lemma}
    
	\section{Characteristic numbers for cubic surfaces, and something more}\label{sec:4}
	In this section we gather all information from Section~\ref{sec:3} in order to compute the characteristic numbers with respect to line conditions for smooth cubic surfaces.
	Recall that for $n=3$ the moduli space of cubic surfaces is $V_0 = \PP^{19}$ and that in Section \ref{sec:2} we constructed a $1$-complete space of cubic surfaces denoted $V_5$. Moreover, $V_{i+1}$ is the blow-up of $V_{i}$ with center $B_i$, and $P_i$ and $L_i$ are the proper transforms, respectively, of point and line conditions in $V_i$.
	Thanks to Theorem \ref{thm:counting} and \cite[Theorem II]{Aluffi:1990} we obtain the following.
	
	\begin{lemma}
	The number $\mathcal{N}(n_p,n_{\ell})$ of smooth cubic surfaces containing $n_p$ given points and tangent to $n_{\ell}$ given lines in general position with $n_p+n_{\ell}=19$ is
	
	\[\mathcal{N}(n_p,n_{\ell})=4^{n_{\ell}}-\sum_{i=0}^4\int_{B_i}\frac{(B_i\circ P_i)^{n_p}(B_i\circ L_i)^{n_{\ell}}}{c(N_{B_i}V_i)}\]
	\end{lemma}
	
	\begin{thm}\label{thm:charNumbers}
	We have  
	\[
	\mathcal{N}(n_p,n_{\ell})=
	\begin{cases}
    4^{19-n_p}, &n_p\in\{7,8,\dots,19\},\\
    67107584, &n_p=6,\\
	268391296, &n_p=5,\\
	1072926016, &n_p=4,\\
	4266198896, &n_p=3,\\
	16615227040, &n_p=2,\\
	61810371328, &n_p=1,\\
	213642327616, &n_p=0.\\
	\end{cases}
	\]
	\end{thm}
	
	\begin{proof}
	The proof is merely computational.
	\[
	\int_{B_0}\frac{(3h)^{n_p}(2+12h)^{n_{\ell}}(1+h)^{4}}{(1+3h)^{20}}=
	\begin{cases}
	1769472, &n_p=3\\
	54263808, &n_p=2\\
	877658112, &n_p=1\\
	9948889088, &n_p=0\\
	\end{cases}
	\]
	\[
	\int_{B_1}\frac{(3h)^{n_p}(1+12h-2\epsilon)^{n_{\ell}}(1+2h-\epsilon)^{10}}{(1+\epsilon)(1+3h-\epsilon)^{20}}=
	\begin{cases}
	434889, &n_p=3\\
	13011156, &n_p=2\\
	203305944, &n_p=1\\
	2199770536, &n_p=0\\
	\end{cases}
	\]
	\[
	\int_{B_2}\frac{(3h)^{n_p}(1+12h-2\epsilon-\phi)^{n_{\ell}}}{(1+\phi)(1+\epsilon-\phi)}=
	\begin{cases}
	17951031, &n_p=3\\
	443328300, &n_p=2\\
	5677810728, &n_p=1\\
	49885157976, &n_p=0\\
	\end{cases}
	\]
	\[
	\int_{B_3}\frac{(l+2m)^{n_p}(1+4l+8m-6e)^{n_{\ell}}}{c(N_{B_3}V_3)}=
	\begin{cases}
	160, &n_p=6\\
	6240, &n_p=5\\
	130224, &n_p=4\\
	1426504, &n_p=3\\
	8284040, &n_p=2\\
	7701512, &n_p=1\\
	-337368096, &n_p=0\\
	\end{cases}
	\]
	\[
	\int_{B_4}\frac{(B_4\circ P_4)^{n_p}(B_4\circ L_4)^{n_{\ell}}}{c(N_{B_4}V_4)} = 
	\begin{cases}
	1120, &n_p=6\\
	37920, &n_p=5\\
	685584, &n_p=4\\
	7186504, &n_p=3\\
    45754840, &n_p=2\\
	142629112, &n_p=1\\
	-460870176, &n_p=0\\
	\end{cases}
	\]
	\end{proof}
	
	\begin{rmk}
	As explained in Remark~\ref{rmk:goeswrong}, we cannot have a similar result for $n>3$. Looking at the computations in 
	\texttt{\href{https://mathrepo.mis.mpg.de/CountingCubicHypersurfaces}{MathRepo}}, it seems that for $n_p\in [n,2n]$ the last correction term is not affected by the ambiguity explained in Remark~\ref{rmk:goeswrong} of the Chern class $c(\mathcal{E})$. We therefore conjecture that, for every $n$, the numbers given by the code for $n_p\geq n$ are the characteristic numbers.
	\end{rmk}
	
    It is an interesting question whether the characteristic numbers above can be attained by numerical algebraic geometry methods. However, the numbers in Theorem \ref{thm:charNumbers} increase fast and it could be numerically challenging to compute them. One could try instead to compute the correction term that need to be subtracted from $4^{19-n_p}$ by numerical software, e.g. \texttt{HomotopyContinuation.jl} \cite{breiding2018homotopycontinuation}.
    We did not pursue this direction, it would be in any case an interesting problem for experts in numerical algebraic geometry. 

	\setcounter{subsection}{-1}
	\subsection{Crumbs of hyperplanes tangency conditions}
	In what follows we stick to the conventions in Section \ref{sec:1}. In the case of hyperplane tangency conditions for hypersurfaces of degree $d$, the base locus is hard to parametrize compared to \eqref{eq:phi0}.
    In fact, hyperplane conditions in $\PP(\Sym^d(W))$ intersect in the locus $B_0^{H}$ of hypersurfaces with positive-dimensional singular locus 
    \[B_0^{H}(d,n) = \{ [h] \in \PP(\Sym^d(W)), \,\, |  \dim \mathcal{V}(h)_{sing}\geq 1 \}.\]
    
    This space has been studied in \cite{slavov2015moduli,tseng2020collections}. In the case of cubic surfaces, \cite{sukarto2020orbit} gives a classification of the cubic surfaces with positive-dimensional singular locus: it consists of the reducible cubics (which forms a variety of dimension $12$) and of the orbit closure under the $\PGL_4$-action of an irreducible cubic surface corresponding to \cite[Table~1,~6C]{sukarto2020orbit}, giving a variety of dimension $13$. In particular, Slavov's theorem \cite[Theorem~1.1]{slavov2015moduli} is true also for cubic surfaces.  
    
    There are some characteristic numbers that we can derive without any complicated construction. For a hyperplane $H \subseteq \PP(W)$, we define $L^H$ to be the variety in $\PP(\Sym^d(W))$ parametrizing all degree $d$ hypersurfaces tangent to~$H$. 
    
    \begin{rmk}
    	The variety $L_H$ of all degree $d$ hypersurfaces in $\PP(W)$ tangent to $H$ has degree $n(d-1)^{n-1}$. 
    	Indeed, assume $H=\mathcal{V}(x_0)$ and $g \in \PP(\Sym^d(W))$. Asking for the hypersurface $g$ to be tangent to $H$ corresponds to the vanishing of the resultant of the polynomials $\partial_{x_i} g(0,x_1,\dots , x_n)$, for $i\in\{1,..,n\}$. This is the resultant of $n$ homogeneous polymonials of degree $d-1$ in $n$ variables, hence it has degree $n(d-1)^{n-1}$.
    \end{rmk}

    Knowing the degree of the variety $L^H$, it is immediate to compute some of the characteristic numbers for degree $d$ hypersurfaces in $\PP(W)$.
    
    \begin{lemma}
    Let $d = 5$, $d \geq 7$ or $(d,n)=(3,3)$. If $n_H<n(d-2)+3$, the number $\mathcal{N}^H(n_p,n_H)$ of degree $d$ (smooth) hypersurfaces in $\PP(W)$ tangent to $n_H$ general hyperplanes and going through $\binom{n+d}{n}-1-n_H$ general points equals $\mathcal{N}^H(n_p,n_H)=(n(d-1)^{n-1})^{n_H}$.
    \end{lemma}
    
    \begin{proof} If we consider $n_H$ hyperplane conditions with $n_H$ strictly less than the codimension of $B_0^{H}(d,n)$, the claim follows from Bézout's theorem. The codimension of $B_0^{H}(d,n)$ is known for $d=5$ or $d\geq 7$ \cite[Theorem 1.6]{tseng2020collections} and arbitrary $n$ with
    \begin{equation*}\label{eq:codim}
    	\hbox{codim} B_0^{H}(d,n) = n(d-2)+3.
    \end{equation*}
    However, this codimension holds true also in the case of cubic surfaces thanks to \cite[Table~1 and Table~2]{sukarto2020orbit}.
    \end{proof}

    \printbibliography	
\end{document}